\documentclass[a4paper,reqno,11pt]{amsart}
\usepackage{mathrsfs}
\usepackage{txfonts}
\usepackage{}
\usepackage{amsfonts}
\usepackage{amssymb}
\usepackage{amsfonts,amsmath,amsthm,amssymb,stmaryrd}

\newtheorem{Theorem}{Theorem}[section]
\newtheorem{Lemma}{Lemma}[section]

\newtheorem{Remark}{Remark}[section]
\newtheorem{Corollary}{Corollary}[section]

\numberwithin{equation}{section}
\usepackage[left=1 in, right=1 in,top=1 in, bottom=1 in]{geometry}

\def\XXint#1#2#3{{\setbox0=\hbox{$#1{#2#3}{\int}$ }
\vcenter{\hbox{$#2#3$ }}\kern-.6\wd0}}

\DeclareMathOperator*{\esssup}{ess\,sup}
\begin{document}
\title[Compressible magneto-micropolar fluid equations ]{Well-posedness of compressible magneto-micropolar fluid equations}
\author{Cuiman Jia}
\address{School of Mathematical Sciences, Xiamen University, Xiamen, Fujian, 361005, P. R. China}
\email[C. M. Jia]{cmjia0124@163.com}
\author{Zhong Tan}
\address{School of Mathematical Sciences and Fujian Provincial Key Laboratory on Mathematical Modeling and Scientific Computing, Xiamen University,
Xiamen, Fujian , 361005 , P. R. China}
\email[Z. Tan]{ztan85@163.com}
\author{Jianfeng Zhou}
\address{School of Mathematical Sciences, Peking University, Beijing 100871, China}
\email[J. Zhou]{jianfengzhou\_xmu@163.com}
\thanks{Corresponding author: Jianfeng Zhou, jianfengzhou\_xmu@163.com}
\thanks{
This work was supported by the National Natural Science Foundation
 of China (No. 11271305, 11531010).}
\begin{abstract}
We are  concerned with compressible magneto-micropolar fluid equations (\ref{1.1})-(\ref{1.2}). The global existence and large time behaviour of
solutions near a constant state to the magneto-micropolar-Navier-Stokes-Poisson (MMNSP) system is investigated in $\mathbb{R}^3$. By a refined energy
method, the global existence is established under the assumption that the $H^3$ norm of the initial data is small, but the higher order derivatives can
be large. If the initial data belongs to homogeneous Sobolev spaces or homogeneous Besov spaces, we prove the optimal time decay rates of the solution
and its higher order spatial derivatives. Meanwhile, we also obtain the usual $L^p-L^2$ $(1\leq p\leq2)$ type of the decay rates without requiring that
the $L^p$ norm of initial data is small.
\bigbreak
\noindent
{\bf \normalsize Keywords }  {Magneto-micropolar fluid;\, Navier-Stokes-Poisson system;\,global existence;\, time decay rate;\,  homogeneous Sobolev space;\, homogeneous Besov space.}\bigbreak
\end{abstract}
\subjclass[2010]{35Q35; 35M31; 76N15; 35B40; 74A35.}

\maketitle
\section{Introduction}
The dynamic of charged particles of one carrier type (e.g., electrons) in the effect of magnetic field can be described by the
magneto-micropolar-Navier-Stokes-Poisson (MMNSP) system:
\begin{equation}\label{1.1}
\begin{cases}
\partial_t\rho+{\rm\ div}~(\rho u)=0, &\\
\partial_t(\rho u)+{\rm\ div}~(\rho u \otimes u)-(\mu+\zeta)\Delta u-(\mu+\lambda-\zeta)\nabla {\rm\ div}~u+\nabla p&\\
=2\zeta \nabla \times w+(\nabla \times b)\times b+\rho\nabla\Phi,&\\
\partial_t(\rho w)+{\rm\ div}~(\rho u \otimes w)-\mu'\Delta w-(\mu'+\lambda')\nabla {\rm\ div}~w+4\zeta w=2\zeta \nabla \times u,&\\
\partial_tb-\nabla\times(u\times b)=-\sigma \nabla \times(\nabla \times b),&\\
\Delta\Phi=\rho-\bar{\rho},&\\
{\rm\ div}~b=0,\quad t>0, x\in\mathbb{R}^3,&
\end{cases}
\end{equation}
with initial data
\begin{equation}\label{1.2}
(\rho,u,b,w)|_{t=0}=(\rho_0,u_0(x),b_0(x),w_0(x))\longrightarrow(\bar{\rho},0,0,0)\quad\text{as}\  |x|\longrightarrow\infty,
\end{equation}
where the unknowns $\rho=\rho(t,x)\geq 0$, $u=(u_1,u_2,u_3)(t,x)$, $w=(w_1,w_2,w_3)(t,x)$, $p=p(\rho)$ and $b=(b_1,b_2,b_3)(t,x)$ stand for the fluid
density, velocity, micro-rational velocity, pressure and magnetic field, respectively. $\rho_0,u_0(x),b_0(x),w_0(x)$ are given and $b_0$ satisfies the
compatibility condition, i.e. $\textrm{div}~b_0=0$. The pressure $p=p(\rho)$ is a smooth function with $p'(\rho)>0$. The parameters
$\mu,\lambda,\zeta,\mu',\lambda'$ and $\sigma$ are constants denoting the viscosity coefficients of the flows satisfying
\begin{equation*}
\mu,\zeta,\mu',\sigma>0, 2\mu+3\lambda-4\zeta\geq 0, 2\mu'+3\lambda'\geq 0.
\end{equation*}
In the motion of the fluid, due to the greater inertia the ions merely provide a constant charged background $\bar{\rho}>0$. In particular, if $w=0$,  then (\ref{1.1})-(\ref{1.2}) reduces to compressible magnetohydrodynamic equations (MHD), which has been studied extensively \cite{pan,hu1,hu,kaw,ume,vol,wjh} and etc. If $w=b=0$, then (\ref{1.1})-(\ref{1.2}) reduces to compressible   Navier-Stokes equations (NS), many work have been done on the existence, $L^p$-decay estimates with $p\geq2$, stability and etc. for either non-isentropic or isentropic case, see e.g. \cite{duan1,duan2,kaw1,kob,lx,mat1,mat2} and the references therein.

The main purpose of this paper is to investigate the influence of the electric field on the time decay rates of the solution to compressible
magneto-micropolar-Navier-Stokes (MMNS) system. We first review some previous works are related to MMNSP system. When $\rho$ be a constant
(e.g. $\bar{\rho}$), then (\ref{1.1})-(\ref{1.2}) reduces to the incompressible MMNS system:
\begin{equation}\label{1.02}
\begin{cases}
\partial_tu+u\cdot\nabla u-\frac{2\zeta}{\bar{\rho}}\nabla\times w-\frac{1}{\bar{\rho}}(\nabla\times b)\times b
-\frac{\mu+\zeta}{\bar{\rho}}\Delta u-\frac{(\mu+\lambda-\zeta)}{\bar{\rho}}\nabla\textrm{div}~u+\nabla P=0,&\\
\partial_tw+u\cdot\nabla w-\frac{\mu'}{\bar{\rho}}\Delta w-\frac{\mu'+\lambda'}{\bar{\rho}}\nabla\textrm{div}~w
+\frac{4\zeta}{\bar{\rho}}w-\frac{2\zeta}{\bar{\rho}}\nabla\times u=0,&\\
\partial_tb-\nabla\times(u\times b)=-\sigma \nabla \times(\nabla \times b),&\\
\textrm{div}~u=\textrm{div}~b=0,\quad t>0,\ x\in\mathbb{R}^3.
\end{cases}
\end{equation}
Such model was first proposed by Galdi-Rionero \cite{twz14}. The system (\ref{1.02}) enable us to study some physical phenomena that can not be treated by
the classical NS equations for the viscous incompressible fluids, e.g., the motion of liquid crystal, animal blood and dilute aqueous
polymer solutions, etc. Due to this important physical background, rich phenomenon, mathematical complex and challenge, there is a lot of literature
devoted to the mathematical theory of (\ref{1.02}). The existence and uniqueness of the strong solution was established by Rojas-Medar \cite{gala18} in
bounded domain $Q_T:=\Omega\times[0,T]$ with $\Omega\subset\mathbb{R}^d$ $(d=2,3)$ and $0<T<\infty$. Further, Rojas-Medar and Boldrini \cite{yao2} proved
the existence of weak solution to (\ref{1.02}) by the Galerkin method. In addition, the authors also showed that such weak solution is unique in two
dimension. Soon after, Rojas-Medar et al. \cite{gala17} derived global in time existence of strong solution for small initial data. In 2010, S. Gala
\cite{gala} established some improved regularity criteria of weak solutions to  (\ref{1.02}) in Morrey-Campanato spaces. Recently, Tan-Wu-Zhou \cite{zhou}
obtained the global existence and large time behaviour of the solutions to (\ref{1.02}) in $\mathbb{R}^3$. Moreover, the authors also derived a weak
solution in $\mathbb{R}^2$ with large initial data. The further literature on the the incompressible MMNS system is indeed huge and thus out of the scope
of this parer, see \cite{silva7,lyj,sadowski,ma,yao,miao19} and the references therein.

Let us  recall important  mathematical characters on  the incompressible MMNS system, the micropolar Navier-Stokes (MNS) system ($b=0$):
\begin{equation}\label{1.03}
\begin{cases}
\partial_t u-(\chi_1+\chi_2)\Delta u+u\cdot\nabla u+\nabla P-2\chi_1\nabla\times w=0,&\\
\partial_t w-\chi_3\Delta w+u\cdot\nabla w+4\chi_1w-\chi_4\nabla\textrm{div}~ w-2\chi_1\nabla\times u=0,&\\
\textrm{div}~ u=0.&
\end{cases}
\end{equation}
The constants $\chi_i$ $(i=1,2,3,4)$ are the viscosity coefficients. MNS system was first developed by Eringen \cite{gala4} in 1966. In any bounded domain
$\Omega\subset\mathbb{R}^d$ ($d=2,3$), Galdi et al. \cite{twz14} and {\L}ukaszewicz \cite{ma4} derived the existence of weak solution. Furthermore, in
\cite{ma4} (see also Yamaguchi \cite{gala23}) the author also proved the existence and uniqueness of strong solution to (\ref{1.03}). For the
well-posedness of MNS system with full viscosity and partial viscosity in $\mathbb{R}^2$, one may refer to \cite{miao10,ma4}, respectively. In 2012, Miao et al. \cite{miao} proved the global well-posedness for the $3D$ MNS system in the critical Besov spaces by making a suitable transformation of the solutions and using the Fourier localization method. For more details, one can refer to \cite{va6,miao18,br12,loa,br16,br30} and the references therein.

For the compressible case, recently, Guo et al. \cite{twz45} studied the global existence and optimal time decay rates of solution to MMNS system
in $\mathbb{R}^3$ by combining the $L^p-L^q$ estimates for the linearized equations and the Fourier splitting method. Later, Wu-Wang \cite{wuw}
gave a pointwise estimates for the $3D$ MNS system, which exhibited generalized Huygen's principle for the fluid density and fluid momentum as the
compressible NS equation. Besides, we would like to refer to \cite{ww7,ww8, ww9,ww21,ww23,ww30,ww31,ww32,ww33} and the references therein.

Without loss of generality, in this paper, we set the constants $\bar{\rho}=1$, $\mu=\zeta=\frac12$, $\lambda=\lambda'=\mu'=\sigma=1$, and note that
\begin{align*}
&(\nabla \times b)\times b=b\cdot \nabla b-\frac{1}{2}\nabla(|b|^2),\\
&\nabla \times \nabla \times b=\nabla \textrm{div}~b-\Delta b,\\
&\nabla \times (u\times b)=(b\cdot \nabla)u-(u\cdot \nabla)b+u~\textrm{div}~b-b~\textrm{div}~u.
\end{align*}
Now, we define $\varrho:=\rho-\bar{\rho}$, from above, then the system (\ref{1.1})-(\ref{1.2}) can be rewritten as
\begin{align}\label{1.3}
\begin{cases}
\partial_t\varrho+\textrm{div}~u=M_1,& \\
\partial_tu+\gamma\nabla\varrho-\Delta u-\nabla \textrm{div}~u- \nabla\times w-\nabla\Phi=M_2,&\\
\partial_tw+2 w-\Delta w-2\nabla \textrm{div}~w-\nabla \times u=M_3,&\\
\partial_tb-\Delta b= M_4,&\\
\Delta\Phi=\varrho,&\\
\textrm{div}~b=0, t>0, x\in \mathbb{R}^3,
\end{cases}
\end{align}
with initial data
\begin{equation}\label{1.4}
(\varrho,u,w,b)(x,0)=(\varrho_0,u_0,w_0,b_0)\longrightarrow(0,0,0,0) \quad \text{as}\  |x|\longrightarrow \infty,
\end{equation}
where the nonlinear terms $M_i$ $(i=1,2,3,4)$ are defined as
\begin{align*}
&M_1=-\textrm{div}~(\varrho u),\\
&M_2=-u\cdot \nabla u-f(\varrho)[\Delta u+\nabla \textrm{div}~u+ \nabla \times w]-h(\varrho)\nabla\varrho
+g(\varrho)[b\cdot \nabla b-\frac{1}{2}\nabla (|b|^2)],\\
&M_3=-u\cdot \nabla w-f(\varrho)[\Delta w+2\nabla \textrm{div}~w-2 w+\nabla \times u],\\
&M_4=b\cdot \nabla u-u\cdot \nabla b-b\textrm{div}~u.
\end{align*}
Here,
\begin{equation}\label{1.5}
\gamma=\frac{p'(1)}{1},~f(\varrho)=\frac{\varrho}{\varrho+1},
~h(\varrho)=\frac{p'(\varrho+1)}{\varrho+1}- \frac{p'(1)}{1},~g(\varrho)=\frac{1}{\varrho+1} .
\end{equation}
For simplicity, in the following, we set $p'(1)=1$, that is $\gamma=1.$

\noindent\textbf{Notation.} Throughout this paper, $c$ denotes a general constant may vary in different estimate. If the dependence need to be explicitly
stressed, some notations like $c_k, c_N$ will be used. We use $c_0$ denotes the constants depending on the initial data $k, N$ and $s$. We may use $A\sim B$,
if there is $c$ and $c'$ such that $cB \leq A \leq c'B$. We will use $a \lesssim b$ if $a \leq cb$. For simplicity, we denote
\begin{equation*}
\| \nabla^l(f,g,h)\|_{H^k}:=\|\nabla^l f\|_{H^k}+\| \nabla^l g\|_{H^k}
+\| \nabla^l h\|_{H^k},
\end{equation*}
and $\int f dx=\int_{\mathbb{R}^3}f dx$. In addition, $\nabla ^l$ with an integer $l \geq 0$ stands for the usual spatial derivations of order $l$. When
$l < 0$ or $l$ is not a positive integer, $\nabla^l$ stands for $\Lambda^l$ defined by
\begin{equation*}
\Lambda^{l}f:=\mathscr{F}^{-1}(|\xi|^l\mathscr{F}f),
\end{equation*}
where $\mathscr{F}$ is the usual Fourier transform operator and $\mathscr{F}^{-1}$ is its inverse. We use $\dot{H}^{s}(\mathbb{R}^3)$, $s\in\mathbb{R}$ to
denote the homogeneous Sobolev spaces on $\mathbb{R}^3$ with norm $\|\cdot\|_{\dot{H}^{s}}$ defined by
$\| f \|_{\dot{H}^{s}}=\|\Lambda^{s}f\|_{L^2}$, and we use
$H^s(\mathbb{R}^3)$ to denote the usual Sobolev spaces with norm $\|\cdot\|_{\dot{H}^{s}}$, and $L^{p}(\mathbb{R}^3)$ $(1 \leq p \leq \infty)$ to
denote the usual $L^p$ spaces with norm $\| \cdot \|_{L^p}$. Finally, we  introduce the homogeneous Besov space,
let $\varphi \in C_{0}^{\infty}(\mathbb{R}_{\xi}^3)$ be a cut-off function such that $\varphi(\xi)=1$ with $|\xi|\leq 1$, and $\varphi(\xi)=0$ with
$|\xi|\leq 2$. Let $\psi(\xi)=\varphi(\xi)-\varphi(2\xi)$ and $\psi_{j}(\xi)=\psi(2^{-j}\xi)$ for $j\in \mathbb{Z}$. Then, by the construction
$\sum_{j\in \mathbb{Z}}\psi_{j}(\xi)=1$ if $\xi \ne 0$, we set $\dot{\Delta}_{j}f=\mathscr{F}^{-1}*f$, then for $s \in \mathbb{R}$, we
define the homogeneous Besov spaces $\dot{B}^s_{p,q}(\mathbb{R}^N)$ with norm $\|\cdot\|_{\dot{B}^s_{p,q}}$  by
\begin{equation*}
\|f\|_{\dot{B}^s_{p,q}(\mathbb{R}^N)}=
\left\{
  \begin{array}{ll}
  \left(\sum_{-\infty<j<\infty}2^{jsq}\|\dot{\Delta}_jf\|_{L^p(\mathbb{R}^N)}^q\right)^{\frac{1}{q}} & 1\leq p\leq\infty,1\leq q<\infty,\\
 \esssup _{j\in \mathbb{Z}} 2^{js} \|\dot{\Delta}_jf\|_{L^p(\mathbb{R}^N)} & 1\leq p\leq\infty,q=\infty,
  \end{array}
\right.
\end{equation*}
For $N \geq 3$, we define the energy functional by
\begin{equation}\label{1.6}
\mathscr{E}_{N}(t):=\mathop{\sum}_{l=0}^{N}\|\nabla^l(\nabla\Phi,\varrho,u,w,b)\|_{L^2}^2,
\end{equation}
and the corresponding dissipation rate by
\begin{equation}\label{1.7}
\mathscr{D}_{N}(t):=\mathop{\sum}_{l=0}^{N}\|\nabla^l(\varrho,w)\|_{L^2}^2+\sum_{l=0}^{N}\|\nabla^{l+1}(\nabla\Phi,u,w,b)\|_{L^2}^2.
\end{equation}
Our main results are stated in the following theorem.
\begin{Theorem}\label{th1.1}
Assume that initial data $(\nabla\Phi(0),\varrho_0,u_0,w_0,b_0) \in H^N $ for an integer $N \geq 3$ and
\begin{equation}\label{1.10}
\int \varrho_0dx=0.
\end{equation}
Then there exists a constant $\delta_0 > 0$ such that if
$\mathscr{E}_{3}(0) \leq \delta_0$, then the problem \eqref{1.3}-\eqref{1.4} admits a unique solution
$(\nabla\Phi,\varrho,u,w,b)\in C([0,\infty];H^3(\mathbb{R}^3))$  satisfying that for
all $t > 0$
\begin{equation}\label{1.11}
\sup_{0 \leq t \leq \infty}\mathscr{E}_{3}(t)+\int_{0}^{\infty}\mathscr{D}_{3}(\tau) d \tau \leq c\mathscr{E}_{3}(0).
\end{equation}
Furthermore, if $\mathscr{E}_{N}(0) < \infty$ for any $N \geq 4$, then (\ref{1.3})-(\ref{1.4}) admits a unique solution
$(\nabla\Phi,\varrho,u,w,b)\in C([0,\infty];H^N(\mathbb{R}^3))$, and for all $t>0$, there holds
\begin{equation}\label{1.12}
\sup_{0 \leq t \leq \infty}\mathscr{E}_{N}(t)+\int_{0}^{\infty}\mathscr{D}_{N}(\tau) d \tau \leq c\mathscr{E}_{N}(0).
\end{equation}
\end{Theorem}
Theorem \ref{th1.1} will be proved in Section \ref{se5} by using the strategy Guo \cite{twz16}. The key point is that, by constructing some interactive
energy
functionals for $k\geq0$
\begin{align*}
&\frac{d}{dt}\sum_{l=k}^{k+1}\int \nabla^l u\cdot \nabla^{l+1}\varrho dx
+\frac{1}{4}\sum_{l=k}^{k+1}(\|\nabla^{l}\varrho\|_{L^2}^2+\|\nabla^{l+1}\varrho\|_{L^2}^2+\|\nabla^{l+1}\nabla\Phi\|_{L^2}^2+
\|\nabla^{l+1}\nabla\Phi\|_{L^2}^2)\nonumber\\
\leq&c_l\|(\varrho,u,w,b)\|_{H^3}\sum_{l=k}^{k+1}(\|\nabla^{l+1}\varrho\|_{L^2}^2+\|\nabla^{l+2}(w,b)\|_{L^2}^2)
+\sum_{l=k}^{k+1}(\|\nabla^{l+1}u\|_{L^2}^2+2\|\nabla^{l+1}w\|_{L^2}^2+4\|\nabla^{l+2}u\|_{L^2}^2),
\end{align*}
then one can derive the dissipative of $\varrho$ and $\nabla\Phi$. This, together with Lemma \ref{le3.1} implies that for $N\geq3$
\begin{equation}\label{1.09}
\frac d{dt}\mathscr{E}_N(t)+\mathscr{D}_N(t)\lesssim \sqrt{\mathscr{E}_3(t)}\mathscr{D}_N(t).
\end{equation}
In virtue of the smallness of $\mathscr{E}_3(0)$ and the argument of Theorem \ref{th3.1}, the Theorem \ref{th1.1} follows from (\ref{1.09}).

In addition, if the initial data belongs to Negative Sobolev or Besov spaces, we can derive some further decay rates of the solution and its higher order
spatial derivatives to system (\ref{1.3})-(\ref{1.4}). Based on the regularity interpolation method developed in Strain-Guo \cite{twz38},
Guo-Wang \cite{twz17} and Sohinger-Strain \cite{twz36}, we can develop a general energy method, that is, by using a family of scaled energy estimates with
minimum derivative counts and interpolation among them, we deduce that
\begin{equation*}
\frac d{dt}\mathscr{E}_k^{k+2}+\mathscr{D}_k^{k+2}\lesssim \mathscr{E}_3(t)\mathscr{D}_k^{k+2}
+\mathscr{E}_3(t)\sum_{l=k}^{k+1}\|\nabla^{l+1}(u,\nabla u,w)\|_{L^2}^2,
\end{equation*}
where $\mathscr{E}_k^{k+2}$ and $\mathscr{D}_k^{k+2}$ are defined by (\ref{5.05}) and (\ref{5.06}). Appealing to the properties of homogeneous Sobolev
space $\dot{H}^{-s}$ or Besov space $\dot{B}^{-s}_{2,\infty}$, we can bound $\|\nabla^k(\nabla\Phi,\varrho,u,w,b)\|^2_{L^2}$ in term of
$\mathscr{E}_k^{k+2}$. Hence, we are able to have
\begin{Theorem}\label{th1.2}
Let all assumptions of Theorem \ref{th1.1} are in force.
Let $(\nabla\Phi,\varrho,u,w,b)(t)$ is the solution to \eqref{1.3}-\eqref{1.4} constructed in Theorem \ref{th1.1} with $N \geq  3$.
Suppose that $(\nabla\Phi(0),\varrho_0,u_0,w_0,b_0)\in \dot{H}^{-s}$ for some $s \in [0, \frac{3}{2})$ or
$(\nabla\Phi(0),\varrho_0,u_0,w_0,b_0)\in \dot{B}_{2,\infty}^{-s}$
for some $s \in (0, \frac{3}{2}]$, then we have
\begin{equation}\label{1.13}
\|(\nabla\Phi,\varrho,u,w,b)(t)\|_{\dot{H}^{-s}} \leq c_0,
\end{equation}
or
\begin{equation}\label{1.14}
\|(\nabla\Phi,\varrho,u,w,b)(t)\|_{\dot{B}_{2,\infty}^{-s}} \leq c_0,
\end{equation}
Moreover, for $k\geq0$, if $N \geq k+2$, there holds
\begin{equation}\label{1.15}
\|\nabla^k(\nabla\Phi,\varrho,u,w,b)(t)\|_{L^2} \leq c_0(1+t)^{-\frac{k+s}{2}},
\end{equation}
and in particular, for $0\leq k\leq N-3$
\begin{equation}\label{1.16}
\|\nabla^k \varrho(t)\|_{L^2}\leq c_0(1+t)^{-\frac{k+1+s}{2}}.
\end{equation}
\end{Theorem}
We shall note that, in the usual $L^p-L^2$ approach of studying the optimal decay rates of the solution, it is difficult to show that the $L^p$-norm
of the solution can be preserved along time evolution. An important feature is that the $\dot{H}^{-s}$ or $\dot{B}^{-s}_{2,\infty}$ norm of the solution
is preserved along time evolution. From Hardy-Littlewood-Sobolev inequality (Lemma \ref{le2.4}), for $p\geq2$, we infer that $L^p\subset\dot{H}^{-s}$
with $s=3(\frac1p-\frac12)\in [0,\frac32)$. This, together with Theorem \ref{th1.2}, we have the $L^p-L^2$ type of the optimal decay results. However, the
imbedding theorem can not cover the case $p=1$. To amend this, Sohinger-Strain \cite{twz36} instead introduced the homogeneous Besov space
 $\dot{B}^{-s}_{2,\infty}$ due to the fact that the endpoint imbedding  $L^1\subset\dot{B}^{-s}_{2,\infty}$ holds (Lemma \ref{le2.5}). At this stage, by
Theorem \ref{th1.2}, we have the following corollary of the usual $L^p-L^2$ type of the decay results:
\begin{Corollary}\label{co1.1}
Under the assumptions of Theorem \ref{th1.2}, except that we replace $\dot{H}^{-s}$ or $\dot{B}^{-s}_{2,\infty}$ assumption by
 $(\nabla\Phi(0),\varrho_0,u_0,b_0,w_0)\in L^p$ for some $p\in [1,2]$. Then, for any integer $k\geq0$, if $N\geq k+2$, then there holds
\begin{equation*}
 \|\nabla^k(\nabla\Phi,\varrho,u,b,w)(t)\|_{L^2}\leq C_0(1+t)^{-\frac{k+s_p}{2}},
\end{equation*}
where $s_p$ is defined by
\begin{equation*}
  s_p:=3(\frac1p-\frac12).
\end{equation*}
\end{Corollary}
The following are several remarks for our main results:
\begin{Remark}
By the Poisson equation, it holds that
\begin{equation*}
\|\varrho_0\|_{H^N\cap\dot{H}^{-s}}+\|\nabla\Phi(0)\|_{H^N\cap\dot{H}^{-s}}\sim
\|\varrho_0\|_{H^N\cap\dot{H}^{-s}}+\|\Lambda^{-1}\varrho_0\|_{H^N\cap\dot{H}^{-s}},\quad s\geq0
\end{equation*}
and
\begin{equation*}
\|\varrho_0\|_{H^N\cap L^p}+\|\nabla\Phi(0)\|_{H^N\cap L^p}\sim
\|\varrho_0\|_{H^N\cap L^p}+\|\Lambda^{-1}\varrho_0\|_{H^N\cap L^p},\quad p\in(1,2].
\end{equation*}
Compared to the study of MMNS system, the norms of $\Lambda^{-1}\varrho_0$ is additionally required. However, such assumption can be achieved by the
natural neutral condition (\ref{1.10}) (cf. \cite{w7}).
\end{Remark}
\begin{Remark}
In Theorem \ref{th1.1}-\ref{th1.2}, if $k=0,1$,  we can remove the smallness of $\|\nabla\Phi(0)\|_{H^3}$ by assuming only $\nabla\Phi(0)\in H^3$. Then,
we can also derive the time decay rates of (\ref{1.15}) for $k=0,1$. Indeed, similar to (\ref{3.03})-(\ref{3.04}), we can re-estimate the
right-most term in (\ref{3.01}) as following. If $l=0$, we have
\begin{align*}
\int\nabla\Phi\cdot (\varrho u)dx &\leq \|\nabla\Phi\|_{L^6}\|\varrho\|_{L^2}\|u\|_{L^3} \\
 & \leq \|u\|_{H^3}\|(\nabla\nabla\Phi,\varrho)\|^2_{L^2}.
\end{align*}
If $l\geq1$, using Lemma \ref{le2.1}, we arrive at
\begin{align*}
\int\nabla^l\nabla\Phi\cdot \nabla^l(\varrho u)dx \leq&\|\nabla^l\nabla\Phi\|_{L^6} \|\nabla^l(\varrho u)\|_{L^{\frac65}}  \\
 \lesssim& \|\nabla^{l+1}\nabla\Phi\|_{L^2}\left(\sum_{s=0}^{[\frac l2]}\|\nabla^s\varrho\|_{L^3}\|\nabla^{l-s}u\|_{L^2}
+\sum_{s=[\frac l2]+1}^{l}\|\nabla^s\varrho\|_{L^2}\|\nabla^{l-s}u\|_{L^3}\right)\\
 \lesssim&\|\nabla^{l+1}\nabla\Phi\|_{L^2}\left(\sum_{s=0}^{[\frac l2]}\|\nabla^{\hat{\alpha}}\varrho\|^{\hat{\theta}}_{L^2}
\|\nabla^l\varrho\|^{1-\hat{\theta}}_{L^2}\|\varrho\|^{1-\hat{\theta}}_{L^2}\|\nabla^{l+1}u\|^{\hat{\theta}}_{L^2}\right.\\
&+\left.\sum_{s=[\frac l2]+1}^{l}\|\varrho\|^{\check{\theta}}_{L^2}\|\nabla^l\varrho\|^{1-\check{\theta}}_{L^2}
\|\nabla^{\check{\alpha}}u\|^{1-\check{\theta}}_{L^2}\|\nabla^{l+1}u\|^{\check{\theta}}_{L^2}\right)\\
\leq& c_l\|(\varrho,u)\|_{H^3}(\|\nabla^{l+1}(\nabla\Phi,u)\|^2_{L^2}+\|\nabla^l\varrho\|^2_{L^2}),
\end{align*}
where
\begin{equation*}
\hat{\theta}=\frac{l-s}{l+1}, \quad \hat{\alpha}=\frac{l+1}{2(l-s)}-1\in (\text{pending}), \quad \check{\theta}=\frac{l-s}{l},\quad
\check{\alpha}=1-\frac{l}{2s}\in (0,\frac12].
\end{equation*}
Here, we may choose $\hat{\alpha}\in [0,3]$, this implies $l=1$. Combining both inequalities and (\ref{3.01}) yields
\begin{equation*}
-\int\nabla^l\nabla\Phi\cdot\nabla^ludx\geq\frac12\frac{d}{dt}\int|\nabla^l\nabla\Phi|^2dx-c_l\|(\varrho,u)\|_{H^3}\|\nabla^{l+1}(\nabla\Phi,u)\|_{L^2}^2
\end{equation*}
with $l=0,1$.
From above, then by a standard argument we can also obtain (\ref{5.04}) and hence derive the arguments of Theorem \ref{th1.1}-\ref{th1.2} for $k=0,1$.
\end{Remark}
\begin{Remark}
Note that both $\dot{H}^{-s}$ and $\dot{B}^{-s}_{2,\infty}$ norms enhance the decay rates of the solution. The constraint $s\leq\frac32$ in Theorem
\ref{th1.2} comes from applying Lemma \ref{le2.4}-\ref{le2.5} to estimate the nonlinear terms in Section \ref{se4}. For $s>\frac32$, the nonlinear
estimates would not work, this in turn restricts $p\geq1$ in Corollary \ref{co1.1}. Furthermore, let $L^2$ decay rate of higher order spatial derivatives
of the solution is obtained in Corollary \ref{co1.1}, then by Lemma \ref{le2.1}, we can deduce that general optimal $L^q$ $(2\leq q\leq\infty)$ time decay
estimates of the solution.
\end{Remark}
The rest of our paper is organized as follows. First of all, in Section \ref{se2}, we present some useful lemma which will be heavily used in our proof.
Next, in Section \ref{se3}, we prove the  local existence of solution $(\nabla\Phi,\varrho,u,w,b)$. For this aim, we first derive the uniform a priori
estimate of solution in subsection \ref{se3.1}. Further, by constructing  the Cauchy sequence, we establish the local existence and the uniqueness of
solution. In Section \ref{se4}, we derive the evolution of the negative Sobolev and Besov norms of solution. Finally, we prove the
main theorem in Section \ref{se5}.
\section{Preliminary}\label{se2}
Throughout this section, we collect some auxiliary results, some of which have been proven elsewhere. First, we will extensively use the Sobolev
interpolation of the Gagliardo-Nirenberg inequality.
\begin{Lemma}\label{le2.1}
Let $0\leq m,\alpha \leq l$, then we have
\begin{equation}\label{2.1}
\|\nabla^{\alpha}f \|_{L^p}\lesssim \| \nabla ^m f \|_{L^q}^{1-\theta}\| \nabla ^l f \|_{L^r}^{\theta},
\end{equation}
where $0 \leq \theta \leq 1$ and $\alpha$ satisfies
\begin{equation*}
\frac{\alpha}{3}-\frac{1}{p}=(\frac{m}{3}-\frac{1}{q})(1-\theta)+(\frac{l}{3}-\frac{1}{r})\theta.
\end{equation*}
Here, when $p=\infty,$ we require that $0 <\theta <1$.
\end{Lemma}
\begin{proof}
See e.g. P.125 in \cite{y60} or Lemma A.1 in \cite{twz17}.
\end{proof}
The following commutator estimate will be used in Section \ref{se3}.
\begin{Lemma}\label{le2.2}
Let $k\geq 1$ be an integer and define the commutator
\begin{equation*}
[\nabla^k,g]h=\nabla^{k}(gh)-g\nabla^{k}h.
\end{equation*}
Then we have
\begin{equation}\label{2.2}
\|[\nabla^k,g]h\|_{L^p}\leq c_k(\| \nabla g\|_{L^{p_1}}\|\nabla^{k-1}h\|_{L^{p_2}}+\|\nabla^k g\|_{L^{p_3}}\|h\|_{L^{p_4}}),
\end{equation}
where $p,p_2,p_3\in (1,+\infty)$ and
\begin{equation*}
\frac1p=\frac1{p_1}+\frac1{p_2}=\frac1{p_3}+\frac1{p_4}.
\end{equation*}
\end{Lemma}
\begin{proof}
If $p=p_2=p_3=2$, then (\ref{2.2}) can be proved by Lemma \ref{le2.1}. For the general cases, one can refer \cite{yj14,wlt}.
\end{proof}
\begin{Lemma}\label{le2.3}
Let $f(\cdot)$ be a smooth function, $w:\mathbb{R}^3\longrightarrow \mathbb{R}^n$ $(n\geq1)$ satisfies $\|w\|_{H^{3}(\mathbb{R}^3)}\leq\delta\ll1.$ If
$f(w)\sim w$ or $f(0)=0$, then for any integer $k\geq 0$, we have
\begin{align}
\|\nabla ^{k}f(w)\|_{L^{\infty}}&\leq c_k  \|\nabla ^{k}w\|_{L^{2}}^{\frac{1}{4}} \|\nabla ^{k+2}w\|_{L^{2}}^{\frac{3}{4}},\label{2.3}\\
&\stackrel{\text{or}}{\leq}c_k\|\nabla^kw\|_{L^{\infty}},\label{2.03}
\end{align}
and
\begin{equation}\label{2.4}
\|\nabla ^{k}f(w)\|_{L^{2}}\leq c_k\|\nabla ^{k}w\|_{L^{2}}.
\end{equation}
\end{Lemma}
\begin{proof}
For the proof of (\ref{2.3}), one can refer Lemma \ref{le3.1} in \cite{twz17}. Here, we only prove (\ref{2.03})-(\ref{2.4}). It is trivial for the case
$k=0$. If
$k\geq1$ by (\ref{2.1}) and the H\"{o}lder's inequality, we get
\begin{align}\label{2.5}
\|\nabla^kf(w)\|_{L^2}&\leq\|\text{a\ sum\ of\ products}\ f^{(q_1,\cdots,q_m)}\nabla^{q_1}w\nabla^{q_2}w\cdots\nabla^{q_m}w\|_{L^2}\nonumber \\
& \lesssim \|\nabla^{q_1}w\|_{L^{2k/q_1}}\|\nabla^{q_2}w\|_{L^{2k/q_2}}\cdots\|\nabla^{q_m}w\|_{L^{2k/q_m}}\nonumber \\
& \lesssim \|\nabla w\|_{L^3}^{1-\frac{q_1}{k}}\|\nabla^kw\|_{L^2}^{\frac{q_1}{k}}\cdots
\|\nabla w\|_{L^3}^{1-\frac{q_m}{k}}\|\nabla^kw\|_{L^2}^{\frac{q_m}{k}}\nonumber \\
& \lesssim  \|w\|_{H^3}^{m-1}\|\nabla^kw\|_{L^2}
\end{align}
with $1\leq m\leq k$, $q_i\geq1$ and $\sum_{i=1}^{m}q_i=k$. This implies (\ref{2.4}).

By the same way, we infer that
\begin{align*}
\|\nabla^kf(w)\|_{L^{\infty}}&\leq\|\text{a\ sum\ of\ products}\ f^{(q_1,\cdots,q_m)}\nabla^{q_1}w\nabla^{q_2}w\cdots\nabla^{q_m}w\|_{L^{\infty}}\nonumber \\
& \lesssim \|\nabla^{q_1}w\|_{L^{\infty}}\|\nabla^{q_2}w\|_{L^{\infty}}\cdots\|\nabla^{q_m}w\|_{L^{\infty}}\nonumber \\
& \lesssim \|\nabla w\|_{L^3}^{1-\frac{q_1}{k}}\|\nabla^kw\|_{L^{\infty}}^{\frac{q_1}{k}}\cdots
\|\nabla w\|_{L^3}^{1-\frac{q_m}{k}}\|\nabla^kw\|_{L^{\infty}}^{\frac{q_m}{k}}\nonumber \\
& \lesssim  \|w\|_{H^3}^{m-1}\|\nabla^kw\|_{L^{\infty}},
\end{align*}
that is (\ref{2.03}).
\end{proof}
If $s\in [0,\frac32)$, the Hardy-Littlewood-Sobolev theorem implies the following $L^p$ type inequality.
\begin{Lemma}\label{le2.4}
Let $0\leq s <\frac{3}{2}, 1<p \leq 2$ with $\frac{1}{2}+\frac{s}{3}=\frac{1}{p},$ then
\begin{equation}\label{2.6}
\| f \|_{\dot{H}^{-s}}\lesssim \| f \|_{L^p}.
\end{equation}
\end{Lemma}
\begin{proof}
See Theorem 1, p.119 in \cite{wyj31} or \cite{wy4} for instance.
\end{proof}
In addition, for $s\in (0.\frac32]$, we will use the following result.
\begin{Lemma}\label{le2.5}
Let $0< s \leq\frac{3}{2}, 1\leq p < 2$ with $\frac{1}{2}+\frac{s}{3}=\frac{1}{p},$ then
\begin{equation}\label{2.7}
\|f\|_{\dot{B}_{2,\infty}^{-s}}\lesssim \| f \|_{L^p}.
\end{equation}
\end{Lemma}
\begin{proof}
See Lemma 4.6 in \cite{twz17}.
\end{proof}
The following two special Sobolev interpolation will be used in the proof of Theorem \ref{th1.2}.
\begin{Lemma}\label{le2.6}
Let $s\geq 0$, and $l\geq 0$, then we have
\begin{equation}\label{2.8}
\|\nabla^l f \|_{L^2}\leq \| \nabla^{l+1} f \|_{L^2}^{1-\theta}\| f \|_{\dot{H}^{-s}}^{\theta},
\end{equation}
where $\theta=\frac{1}{l+1+s}$.
\end{Lemma}
\begin{proof}
It follows directly by the Parseval theorem and the H\"{o}lder's inequality.
\end{proof}
\begin{Lemma}\label{le2.7}
Let $s >0$,and $l \geq 0$, then we have
\begin{equation}\label{2.9}
\|\nabla^l f \|_{L^2}\leq \|\nabla^{l+1}f\|_{L^2}^{1-\theta}\|f\|_{\dot{B}_{2,\infty}^{-s}}^{\theta},
\end{equation}
where $\theta=\frac{1}{l+1+s}$.
\end{Lemma}
\begin{proof}
See Lemma 4.5 in \cite{twz36}.
\end{proof}
\section{Energy estimate and the local existence}\label{se3}
\subsection{Uniform a priori estimate}\label{se3.1}
Theorem \ref{th1.1} will be proved by combining the local existence of $(\nabla\Phi,\varrho,u,w,b)$ to (\ref{1.3})-(\ref{1.4}) and some uniform a priori
estimates as well as the communication argument. In this section, we aim to derive the uniform a priori estimates.
\begin{Lemma}\label{le3.1}
Let $T>0$. Suppose that
\begin{align}\label{3.1}
\sup_{0\leq t\leq T}\| (\nabla\Phi,\rho,u,w,b)\|_{H^3}\leq \delta,
\end{align}
for $0 \leq \delta \leq 1,$ let all assumptions of Theorem \ref{th1.1} are in force. Then for any integer $k\geq0$ and $t\geq0$, there holds
\begin{align}\label{3.2}
&\frac{1}{2}\frac{d}{dt}\sum_{l=k}^{k+2}\|\nabla^l(\nabla\Phi,\varrho,u,w,b)\|_{L^2}^{2}+\frac12\sum_{l=k}^{k+2}\|\nabla^lw\|_{L^2}^{2}
+\sum_{l=k}^{k+2}(\frac23\|\nabla^{l+1}u\|_{L^2}^{2}+3\|\nabla^{l+1}w\|_{L^2}^{2}+\|\nabla^{l+1}b\|_{L^2}^{2})\nonumber\\
&\leq c_l\delta\sum_{l=k}^{k+2}\left(\|\nabla^l(\varrho,w)\|_{L^2}^{2}+\|\nabla^{l+1}(\nabla\Phi,u,w,b)\|_{L^2}^{2}\right).
\end{align}
\end{Lemma}
\begin{proof}
For any integer $k\geq0$, by the $\nabla ^{l}$ $(l=k,k+1,k+2)$ energy estimate, for   $\eqref{1.3}_1$-$\eqref{1.3}_2$ on $\varrho$ and $u$, we have
\begin{align}\label{3.3}
\frac{1}{2}&\frac{d}{dt}\int|\nabla ^{l}\varrho|^2+|\nabla ^{l}u|^2dx+2\int|\nabla ^{l+1}u|^2dx-\int\nabla^l\nabla\Phi\cdot\nabla^ludx\nonumber\\
=&\int\big[-\nabla ^{l}(\textrm{div}~(\varrho u))\cdot \nabla^{l}\varrho
+\nabla^{l}(\nabla \times w)\cdot \nabla^{l}u-\nabla^{l}(u\cdot\nabla u)\cdot\nabla^{l}u\nonumber\\
&-\nabla^{l}(f(\varrho)\Delta u)\cdot \nabla^{l}u-\nabla^{l}(f(\varrho)\nabla \textrm{div}~u)\cdot \nabla^{l}u
-\nabla^{l}(f(\varrho)\nabla \times w)\cdot\nabla^{l}u\nonumber\\
&-\nabla^{l}(h(\varrho)\nabla\varrho)\cdot\nabla^{l}u
+\nabla^{l}(g(\varrho)b\cdot\nabla b)\cdot \nabla^{l}u-\frac{1}{2}\nabla^{l}(g(\varrho)\nabla (|b|^2))\cdot\nabla^{l}u\nonumber\\
&-\nabla^{l}(\textrm{div}~u)\cdot\nabla^{l}\varrho- \nabla^{l}(\nabla\varrho)\cdot \nabla^{l}u\big]dx\nonumber\\
:=&\mathop{\sum}_{i=1}^{11}I_{i}.
\end{align}
For the last term on the left hand side of of (\ref{3.3}), by integration by parts, (\ref{1.3})$_1$ and the Poisson equation (\ref{1.3})$_5$, we deduce
that
\begin{align}\label{3.01}
-\int\nabla^l\nabla\Phi\cdot\nabla^ludx&=\int\nabla^l\Phi\cdot\nabla^l {\rm div}~udx\nonumber   \\
&=-\int\nabla^l\Phi\cdot\nabla^l\partial_t\varrho+\nabla^l\Phi\cdot\nabla^l{\rm div}~(\varrho u)dx\nonumber   \\
&=-\int\nabla^l\Phi\cdot\nabla^l\partial_t\Delta\Phi-\nabla^l\nabla\Phi\cdot\nabla^l(\varrho u)dx\nonumber   \\
&=\frac12\frac{d}{dt}\int|\nabla^l\nabla\Phi|^2dx+\int\nabla^l\nabla\Phi\cdot\nabla^l(\varrho u)dx.
\end{align}
In addition, by  integration by parts, the H\"{o}lder's inequality and Lemma \ref{le2.1}, it holds that
\begin{align}\label{3.02}
\int\nabla^l\nabla\Phi\cdot\nabla^l(\varrho u)dx&=\int\nabla^l(\Delta\Phi u)\cdot\nabla^l\nabla\Phi dx\nonumber   \\
 &=-\int\nabla^l(\nabla\Phi\cdot\nabla u)\cdot\nabla^l\nabla\Phi+\nabla^l(\nabla\Phi\cdot u)\cdot(\nabla\cdot\nabla^l\nabla\Phi) dx\nonumber \\
&=-\int\sum_{s=0}^{l}C_l^s\left((\nabla^s\nabla\Phi\cdot\nabla^{l-s}\nabla u)\cdot\nabla^l\nabla\Phi
+(\nabla^s\nabla\Phi\cdot\nabla^{l-s}u)\cdot(\nabla\cdot\nabla^l\nabla\Phi)\right)dx\nonumber \\
&\lesssim\sum_{s=0}^l\|\nabla^s\nabla\Phi\|_{L^3}(\|\nabla^{l+1-s}u\|_{L^2}\|\nabla^l\nabla\Phi\|_{L^6}
+\|\nabla^{l-s}u\|_{L^6}\|\nabla\nabla^l\nabla\Phi\|_{L^2})\nonumber \\
&\lesssim\sum_{s=0}^l\|\nabla^s\nabla\Phi\|_{L^3}\|\nabla^{l+1-s}u\|_{L^2}\|\nabla^{l+1}\nabla\Phi\|_{L^2}.
\end{align}
Making use of Lemma \ref{le2.1}, we can see that
\begin{align}\label{3.03}
\sum_{s=0}^l\|\nabla^s\nabla\Phi\|_{L^3}\|\nabla^{l+1-s}u\|_{L^2}\lesssim &
\sum_{s=0}^{[\frac l2]}\|\nabla^{\bar{\alpha}}\nabla\Phi\|_{L^2}^{\bar{\theta}}\|\nabla^{l+1}\nabla\Phi\|_{L^2}^{1-\bar{\theta}}
\|u\|_{L^2}^{1-\bar{\theta}}\|\nabla^{l+1}u\|_{L^2}^{\bar{\theta}}\nonumber  \\
&+\sum_{[\frac l2]+1}^l\|\nabla\Phi\|_{L^2}^{\tilde{\theta}}\|\nabla^{l+1}\nabla\Phi\|_{L^2}^{1-\tilde{\theta}}
\|\nabla^{\tilde{\alpha}}u\|_{L^2}^{1-\tilde{\theta}}\|\nabla^{l+1}u\|_{L^2}^{\tilde{\theta}},
\end{align}
where $\bar{\alpha},\tilde{\alpha},\bar{\theta},\tilde{\theta}$ satisfy
\begin{equation*}
\bar{\theta}=1-\frac{s}{l+1},\ \bar{\alpha}=\frac{k+1}{2(k+1-s)}\in [\frac12,1),\ \tilde{\theta}=1-\frac{2s+1}{2(l+1)},\
\tilde{\alpha}=\frac{l+1}{2s+1}\in (\frac12,1).
\end{equation*}
From (\ref{3.01})-(\ref{3.03}), it follows that
\begin{equation}\label{3.04}
-\int\nabla^l\nabla\Phi\cdot\nabla^ludx\geq\frac12\frac{d}{dt}\int|\nabla^l\nabla\Phi|^2dx
-c\|(\nabla\Phi,u)\|_{H^3}\|\nabla^{l+1}(\nabla\Phi,u)\|_{L^2}^2.
\end{equation}
Now, we concentrate our attention on  estimating the  terms $I_{1}-I_{11}$. First, for the term $I_1$, we can see that
\begin{align}\label{3.4}
I_1&=-\int \nabla ^{l}(\varrho\textrm{div}~u)\cdot \nabla^{l}\varrho dx
-\int \nabla ^{l}(u \cdot \nabla \varrho)\cdot \nabla^{l}\varrho dx\nonumber\\
&:=I_{1a}+I_{1b}.
\end{align}
If $l=0$, it is obvious that
\begin{equation}\label{3.05}
I_{1a}\leq \|\varrho\|_{H^3}\|(\varrho,\nabla u)\|_{L^2}^2.
\end{equation}
If $l=1$, we further obtain
\begin{align}\label{3.06}
I_{1a}&=-\int(\nabla\varrho{\rm div}~u)\cdot\nabla\varrho dx-\int(\varrho \nabla{\rm div}~u)\cdot\nabla\varrho\nonumber \\
&\leq \|\nabla\varrho\|_{L^3}\|\nabla\varrho\|_{L^2}\|{\rm div}~u\|_{L^6}
+\|\varrho\|_{L^{\infty}}\|\nabla{\rm div}~u\|_{L^2}\|\nabla\varrho\|_{L^2}\nonumber \\
&\leq c\|\varrho\|_{H^3}\|(\nabla\varrho,\nabla^2 u)\|_{L^2}^2.
\end{align}
If $l\geq2$, by the Leibniz formula and the H\"{o}lder's inequality, we obtain
\begin{align}\label{3.5}
I_{1a}&=-\int \sum_{s=0}^{l}C_{l}^{s}\nabla ^{s}\varrho \nabla^{l-s} \textrm{div}~u \cdot \nabla ^{l}\varrho dx\nonumber \\
&\lesssim\sum_{s=0}^{[\frac{l}{2}]}\| \nabla ^{s}\varrho  \nabla^{l-s+1} u\|_{L^2}\| \nabla ^{l} \varrho\|_{L^2}
+\sum_{s=[\frac{l}{2}]+1}^{l}\| \nabla ^{s}\varrho  \nabla^{l-s+1} u\|_{L^2}\| \nabla ^{l} \varrho\|_{L^2}\nonumber \\
&\lesssim\big( \sum_{s=0}^{[\frac{l}{2}]}\| \nabla ^{s}\varrho\|_{L^{\infty}}  \|\nabla^{l-s+1} u\|_{L^2}
+\sum_{s=[\frac{l}{2}]+1}^{l}\| \nabla ^{s}\varrho\|_{L^2} \|\nabla^{l-s+1} u\|_{L^{\infty}}\big)\| \nabla ^{l} \varrho\|_{L^2}\nonumber \\
&:=(I_{1aa}+I_{1ab})\|\nabla ^{l} \varrho\|_{L^2},
\end{align}
Making use of (\ref{2.1}), it holds that
\begin{align}\label{3.6}
\notag I_{1aa}&\lesssim\|\nabla^{\alpha}\varrho\|_{L^2}^{\theta}\|\nabla^{l}\varrho\|_{L^2}^{1-\theta}
\|\nabla u\|_{L^2}^{1-\theta} \|\nabla^{l+1}u\|_{L^2}^{\theta}\\
&\leq c_l \|(\varrho,u)\|_{H^{3}}\|(\nabla ^{l}\varrho,\nabla^{l+1}u)\|_{L^{2}},
\end{align}
where
\begin{align*}
    \theta=\frac{l-s}{l}\in (\frac{1}{2},1),\quad \alpha=\frac{3l}{2(l-s)}\in (\frac{3}{2},3].
\end{align*}
Likewise, the term $I_{1ab}$ can be estimated as
\begin{align}\label{3.7}
\notag I_{1ab}&\lesssim \|\nabla^{l}\varrho\|_{L^2}^{\theta_1}\|\nabla^{l}\varrho\|_{L^2}^{1-\theta_1}
\|\nabla^{\alpha_1} u\|_{L^2}^{1-\theta_1} \|\nabla^{l+1}u\|_{L^2}^{\theta_1}\nonumber\\
&\leq c_l \|(\varrho,u)\|_{H^{3}}\|(\nabla ^{l}\varrho,\nabla^{l+1}u)\|_{L^{2}},
\end{align}
where
\begin{align*}
\theta_1=\frac{l-s}{l-1},\quad \alpha_1=\frac{l-1}{2(s-1)}+2\in [\frac{5}{2},3].
\end{align*}
Inserting \eqref{3.6}-\eqref{3.7} into \eqref{3.5}, and together with (\ref{3.05})-(\ref{3.06}), we have
\begin{equation}\label{3.8}
I_{1a}\leq c_l\|(\varrho,u)\|_{H^3}(\|\nabla^{l}\varrho\|_{L^2}^{2}+\|\nabla^{l+1}u\|_{L^2}^{2}).
\end{equation}
Employing Lemma \ref{le2.2}, for the term $I_{1b}$, we infer that
\begin{align}\label{3.9}
I_{1b}&=-\int (u\cdot \nabla\nabla^{l}\varrho+[\nabla ^{l},u]\cdot \nabla \varrho)\cdot \nabla^{l}\varrho dx\nonumber\\
&\lesssim \|\nabla u\|_{L^{\infty}}\|\nabla^{l}\varrho\|_{L^2}^{2}
+(\|\nabla u\|_{L^{\infty}}\|\nabla^{l}\varrho\|_{L^2}^{2}+\|\nabla^{l} u\|_{L^{6}}\|\nabla\varrho\|_{L^3}\|\nabla^{l}\varrho\|_{L^2})\nonumber\\
&\leq c_l\|(\varrho,u)\|_{H^3}(\|\nabla^l \varrho\|_{L^2}^2+\| \nabla^{l+1}u\|_{L^2}^2).
\end{align}
Plugging \eqref{3.8}-\eqref{3.9} into \eqref{3.4}, yields that
\begin{equation}\label{3.10}
I_1\leq c_l\|(\varrho,u)\|_{H^3}(\|\nabla^{l}\varrho\|_{L^2}^2+\|\nabla^{l+1}u\|_{L^2}^2).
\end{equation}
Next, in virtue of the  H\"{o}lder and Sobolev's inequality, we are in a position to obtain
\begin{align}\label{3.11}
I_3&=-\sum_{s=0}^{l}C_{l}^{s}\int (\nabla^s u\cdot \nabla \nabla ^{l-s}u)\cdot \nabla ^{l}u dx\nonumber\\
&\lesssim \|\sum_{s=0}^{l}\nabla ^{s} u\cdot\nabla^{l+1-s}u\|_{L^{\frac{6}{5}}}\|\nabla^lu\|_{L^6}\nonumber\\
&\lesssim (\sum_{s=0}^{[\frac{l}{2}]}\| \nabla^{s}u\cdot \nabla^{l+1-s}u\|_{L^{\frac{6}{5}}}
+\sum_{s=[\frac{l}{2}]+1}^{l}\| \nabla^{s}u\cdot \nabla^{l+1-s}u\|_{L^{\frac{6}{5}}})\| \nabla^{l+1} u\|_{L^2}\nonumber\\
&:=(I_{3a}+I_{3b})\|\nabla^{l+1}u\|_{L^2}.
\end{align}
Making use of \eqref{2.1}, we deduce that
\begin{align}\label{3.12}
I_{3a}&\lesssim \| \nabla^s u\|_{L^3}\| \nabla^{l+1-s}u\|_{L^2}\nonumber\\
&\lesssim\|\nabla^{\alpha_2}u\|_{L^2}^{\theta_2}\| \nabla^{l+1}u\|_{L^2}^{1-\theta_2}
\|u\|_{L^2}^{1-\theta_2}\|\nabla^{l+1}u\|_{L^2}^{\theta_2}\nonumber\\
&\leq c_l\|u\|_{H_1}\|\nabla^{l+1}u\|_{L^2},
\end{align}
where
\begin{equation*}
\theta_2=\frac{l+1-s}{l+1},\quad \alpha_2=\frac{l+1}{2(l+1-s)}\in [\frac{1}{2},1).
\end{equation*}
Similarly, we also have
\begin{align}\label{3.13}
I_{3b}&\lesssim \| \nabla^s u \|_{L^2}\|\nabla^{l+1-s}u\|_{L^3}\nonumber\\
&\lesssim \| u\|_{L^2}^{\theta_3}\|\nabla^{l+1}u\|_{L^2}^{1-\theta_3}
\|\nabla^{\alpha_3}u\|_{L^2}^{1-\theta_3}\|\nabla^{l+1}u\|_{L^2}^{\theta_3}\nonumber\\
&\leq c_l \|u\|_{H^1}\|\nabla^{l+1}u\|_{L^2},
\end{align}
where
\begin{equation*}
\theta_3=\frac{l+1-s}{l+1},\quad \alpha_3=\frac{l+1}{2s}\in (\frac{1}{2},1].
\end{equation*}
As a consequence, from  \eqref{3.11}-\eqref{3.13}, it follows that
\begin{equation}\label{3.14}
I_3\leq c\| u \|_{H^1}\| \nabla^{l+1} u\|^2_{L^2}.
\end{equation}

Now, we estimate the term $I_4$. If $l=0$, appealing to the H\"{o}lder's inequality, \eqref{2.1}, \eqref{2.4} and Cauchy's inequality, we obtain
\begin{align}\label{3.15}
I_4&=\|f(\varrho)\|_{L^2}\|\Delta u\|_{L^3}\|u\|_{L^6}\nonumber\\
&\lesssim \| u\|_{H^3}(\| \varrho \|_{L^2}^2+\| \nabla u \|_{L^2}^2).
\end{align}
If $l=1$, by integration by parts, Lemma \ref{le2.3} and Lemma \ref{le2.1}, we can see that
\begin{align}\label{3.16}
I_4&=\int f(\varrho)|\Delta u|^2 dx \nonumber\\
&\lesssim \|f(\varrho)\|_{L^6}\|\Delta u\|_{L^3}\|\Delta u\|_{L^2}\nonumber\\
&\lesssim \|u\|_{H^3}(\|\nabla\varrho\|_{L^2}^2+\|\nabla^2 u\|_{L^2}^2).
\end{align}
Similarly, if $l=2$, using Lemma \ref{le2.3} and Lemma \ref{le2.1}, we obtain
\begin{align}\label{3.17}
I_4&=\int (\nabla f(\varrho)\Delta u+f(\varrho)\nabla \Delta u)\cdot \textrm{div}\nabla^2 u dx\nonumber\\
&\lesssim \| \nabla f(\varrho)\|_{L^6}\| \Delta u \|_{L^3}\| \nabla^3 u\|_{L^2}+\| f(\varrho)\|_{L^{\infty}}\|\nabla^3 u\|_{L^2}^{2}\nonumber\\
&\lesssim \| (\varrho,u)\|_{H^3}(\| \nabla^2 \varrho\|_{L^2}^2+\| \nabla^3 u\|_{L^2}^2).
\end{align}
If $l\geq3$, by integration by parts and the H\"{o}lder's inequality, we deduce that
\begin{align}\label{3.18}
I_4&=\int\nabla^{l-1}(f(\varrho)\Delta u)\cdot\textrm{div}\nabla^l u dx\nonumber\\
&=\sum_{s=0}^{l-1}C_{l-1}^{s}
\int(\nabla^{s}f(\varrho)\nabla^{l-1-s}\Delta u)\cdot \textrm{div}\nabla ^ludx\nonumber\\
&\lesssim \sum_{s=0}^{l-1}\|\nabla^sf(\varrho)\nabla^{l-1-s}\Delta u\|_{L^2}\|\nabla^{l+1}u\|_{L^2}.
\end{align}
If $0 \leq s \leq [\frac{l-1}{2}]$, by (\ref{2.03}) and Lemma \ref{le2.1}, we arrive at
\begin{align}\label{3.19}
\|\nabla^s f(\varrho)\nabla ^{l-1-s}\Delta u\|_{L^2}&\leq\|\nabla^{s}f(\varrho)\|_{L^{\infty}} \|\nabla^{l+1-s} u\|_{L^2}\nonumber\\
&\lesssim \| \nabla^s \varrho \|_{L^{\infty}}\|\nabla^{l+1-s} u\|_{L^2}\nonumber\\
&\lesssim \|\nabla^{\alpha_4}\varrho\|_{L^2}^{\theta_4}\|\nabla^{l}\varrho\|_{L^2}^{1-\theta_4}
\|\Delta u\|_{L^2}^{1-\theta_4}\|\nabla^{l+1} u\|_{L^2}^{\theta_4}\nonumber\\
&\leq c_l \|(\varrho,u)\|_{H^3}(\| \nabla^{l} \varrho \|_{L^2}+\|\nabla^{l+1}u\|_{L^2}),
\end{align}
where
\begin{align*}
\theta_4=\frac{l-1-s}{l-1}\in (\frac{1}{2},1],\quad \alpha_4=\frac{3}{2}+\frac{s}{2(l-1-s)}\in (\frac{3}{2},2].
\end{align*}
By the same way, if $[\frac{l-1}{2}]+1\leq s \leq l-1$, in virtue of \eqref{2.1} and \eqref{2.4}, we can see that
\begin{align}\label{3.20}
\|\nabla^{s} f(\varrho)\nabla^{l-1+s}\Delta u \|_{L^2}&\leq\|\nabla^{s} f(\varrho)\|_{L^2}\|\nabla^{l+1-s}u\|_{L^{\infty}}\nonumber\\
&\lesssim\|\nabla^{s}\varrho\|_{L^2}\|\nabla^{l+1-s}u\|_{L^{\infty}}\nonumber\\
&\lesssim\|\nabla^{2}\varrho\|_{L^2}^{\theta_5}\|\nabla^{l}\varrho\|_{L^2}^{1-\theta_5}
\|\nabla^{\alpha_5}u\|_{L^2}^{1-\theta_5}\|\nabla^{l+1} u \|_{L^2}^{\theta_5}\nonumber\\
&\leq c_l \|(\varrho,u)\|_{H^3}(\| \nabla^l \varrho\|_{L^2}+\|\nabla^{l+1}u\|_{L^2}),
\end{align}
where
\begin{equation*}
\theta_5=\frac{l-s}{l-2}\in (0,1),\quad \alpha_5=3-\frac{l-2}{2(s-2)}\in (2,3).
\end{equation*}
Inserting \eqref{3.19}-\eqref{3.20} into \eqref{3.18}, by the Cauchy's inequality, we obtain that for $l\geq3$,
\begin{equation}\label{3.21}
I_4\leq c_l\|(\varrho,u)\|_{H^3}(\|\nabla^{l}\varrho\|_{L^2}^2+\|\nabla^{l+1}u\|_{L^2}^2).
\end{equation}
Together \eqref{3.15}-\eqref{3.17} with \eqref{3.21}, we finally obtain for $l\geq 0$
\begin{equation}\label{3.22}
I_4\leq c_l\|(\varrho,u)\|_{H^3}(\|\nabla^l\varrho\|_{L^2}^2+\|\nabla^{l+1}u\|_{L^2}^2).
\end{equation}
Similarly, for the term $I_5$, we deduce that for $l\geq0$
\begin{equation}\label{3.23}
I_5\leq c_l\|(\varrho,u)\|_{H^3}(\|\nabla^{l}\varrho\|_{L^2}^2+\|\nabla^{l+1}u\|_{L^2}^2).
\end{equation}
Now, we estimate the term $I_6$, similar to \eqref{3.15}-\eqref{3.16}, if $l=0$, making use of \eqref{2.4} and Lemma \ref{le2.1}, we can see that
\begin{align}\label{3.24}
I_6&=-\int(f(\varrho)\nabla \times w)\cdot u dx\nonumber\\
&\lesssim \|\varrho\|_{L^2}\|\nabla\times w\|_{L^3}\| u \|_{L^6}\nonumber\\
&\lesssim \|\varrho\|_{L^2}(\|\nabla w\|_{L^2}+\| w \|_{H^3})\| \nabla u \|_{L^2}\nonumber\\
&\lesssim \|(\varrho,w)\|_{H^3}(\| \varrho \|_{L^2}^2+\|(\nabla w,\nabla u)\|_{L^2}^2).
\end{align}
If $l=1$, by integration by parts, the  H\"{o}lder's inequality, Lemma \ref{le2.1} and Lemma \ref{le2.3}, we infer that
\begin{align}\label{3.25}
I_6&=\int(f(\varrho)\nabla \times w)\cdot \textrm{div}\nabla u dx\nonumber\\
&\lesssim \| f(\varrho)\|_{L^3}\| \nabla w\|_{L^6}\|\nabla^2 u\|_{L^2}\nonumber\\
&\lesssim \|\varrho\|_{H^3}\|(\nabla^2 w,\nabla^2 u)\|_{L^2}^2.
\end{align}
Next, for $l\geq 2$, similar to \eqref{3.18}-\eqref{3.20}, there holds
\begin{align}\label{3.26}
I_6&=\sum_{s=0}^{l-1}\int (\nabla^s f(\varrho)\nabla^{l-1-s}\nabla\times w)\cdot \textrm{div}\nabla^l u dx\nonumber\\
&\leq \| \sum_{s=0}^{l-1}(\nabla^{s}f(\varrho)\nabla^{l-1-s}\nabla \times w)\|_{L^2}\|\nabla^{l+1}u\|_{L^2}.
\end{align}
If $0\leq s\leq\frac{[l-1]}{2}$, making use of the H\"{o}lder's inequality, \eqref{2.1}, we obtain
\begin{align}\label{3.27}
\|\nabla^s f(\varrho)\nabla^{l-1-s}\nabla\times w\|_{L^2}&\leq \| \nabla^s f(\varrho)\|_{L^3}\|\nabla^{l-1-s}\nabla\times w\|_{L^6}\nonumber\\
&\lesssim\|\nabla^{\alpha_6}\varrho \|_{L^2}^{\theta_6}\|\nabla^l \varrho\|_{L^2}^{1-\theta_6}
\|\nabla w\|_{L^2}^{1-\theta_6}\|\nabla ^{l+1}w\|_{L^2}^{\theta_6}\nonumber\\
&\lesssim \|(\varrho,w)\|_{H^3}(\|\nabla^l \varrho\|_{L^2}+\|\nabla^{l+1}w\|_{L^2}),
\end{align}
where
\begin{equation*}
\theta_6=\frac{l-s}{l},\quad \alpha_6=\frac{l}{2(l-s)}\in [\frac{1}{2},1).
\end{equation*}
If $[\frac{l-1}{2}]+1\leq s\leq l-1$, applying Lemma \ref{le2.1} and Lemma \ref{le2.3}, there holds
\begin{align}\label{3.28}
\|\nabla^s f(\varrho)\nabla^{l-1-s}\nabla \times w\|_{L^2}&\leq \|\nabla^s f(\varrho)\|_{L^6}\|\nabla^{l-1-s}\nabla \times w\|_{L^3}\nonumber\\
&\lesssim \|\nabla^{s+1}\varrho\|_{L^2}\|\nabla^{l-s}w\|_{L^3}\nonumber\\
&\lesssim \|\nabla \varrho\|_{L^2}^{\theta_7}\|\nabla^l \varrho\|_{L^2}^{1-\theta_7}
\|\nabla^{\alpha_7}w\|_{L^2}^{1-\theta_7}\|\nabla^{l+1}w\|_{L^2}^{\theta_7}\nonumber\\
&\leq c_l\|(\varrho,w)\|_{H^3}(\|\nabla^{l} \varrho \|_{L^2}^2+\|\nabla^{l+1} w \|_{L^2}^2),
\end{align}
where
\begin{equation*}
\theta_7=\frac{l-1-s}{l-1},\quad \alpha_7=2-\frac{l-1}{2s}\in (1,\frac{3}{2}].
\end{equation*}
Plugging \eqref{3.27}-\eqref{3.28} into \eqref{3.26}, we deduce that for $l\geq 2$,
\begin{equation}\label{3.29}
I_6 \leq c_l \|(\varrho,w)\|_{H^3}(\|\nabla^l \varrho\|_{L^2}^2+\|(\nabla^{l+1}w,\nabla^{l+1}u)\|_{L^2}^2).
\end{equation}
Thus, combining \eqref{3.24}-\eqref{3.25} and \eqref{3.29}, we conclude that for $l\geq 0$
\begin{equation}\label{3.30}
I_6\leq c_l\|(\varrho,w)\|_{H^3}(\| \nabla^l \varrho\|_{L^2}^2+\|(\nabla^{l+1}w,\nabla^{l+1}u)\|_{L^2}^2).
\end{equation}
Noting that $h(\varrho)$ is a smooth function, similar to \eqref{3.30}, the term $I_7$ can be estimated as
\begin{equation}\label{3.31}
I_7\leq c_l\|\varrho\|_{H^3}(\|\nabla^l \varrho\|_{L^2}^2+\|\nabla^{l+1}u\|_{L^2}^2).
\end{equation}
In addition, observe that
\begin{equation*}
g(\varrho)=1-f(\varrho), \quad \|\nabla^l(b\otimes b)\|_{L^2}\lesssim \| b \|_{H^3}\|\nabla^{l+1} b\|_{L^2}.
\end{equation*}
Thus, by integration by parts and the H\"{o}lder's inequality, similar to the estimate of $I_6$, the term $I_8$ can be estimated as
\begin{align}\label{3.32}
I_8&=-\int (g(\varrho)\nabla^{l-1}\nabla\cdot(b\otimes b))\cdot (\nabla\cdot\nabla^{l}u) dx
+\sum_{s=1}^{l-1}C_l^s\int (\nabla^s f(\varrho)\nabla^{l-1-s}\nabla\cdot(b\otimes b))\cdot(\nabla\cdot\nabla^{l}u) dx\nonumber\\
&\lesssim \|\nabla^l (b\otimes b)\|_{L^2}\|\nabla^{l+1}u\|_{L^2}
+\sum_{s=1}^{l-1}\|\nabla^s f(\varrho)\cdot \nabla^{l-1-s}\nabla \cdot(b\otimes b)\|_{L^2}\|\nabla^{l+1}u\|_{L^2}\nonumber\\
&\leq c_l\|(\varrho,b)\|_{H^3}(\|\nabla^l \varrho\|_{L^2}^2+\|(\nabla^{l+1}b,\nabla^{l+1}u)\|_{L^2}^2).
\end{align}
Likewise, for the term $I_9$, we can also obtain
\begin{equation}\label{3.33}
I_9\leq c_l \|(\varrho,b)\|_{H^3}(\|\nabla^l \varrho\|_{L^2}^2+\|(\nabla^{l+1}b,\nabla^{l+1}u)\|_{L^2}^2).
\end{equation}
Taking into account (\ref{3.04}), \eqref{3.10}, \eqref{3.14}, \eqref{3.22}, \eqref{3.23}, \eqref{3.30}-\eqref{3.33} and note that $I_{10}+I_{11}=0$,
we finally obtain
\begin{align}\label{3.34}
&\frac{1}{2}\frac{d}{dt}\int(|\nabla^l\nabla\Phi|^2+|\nabla^l \varrho|^2+|\nabla^l u|^2)dx+2\int |\nabla^{l+1}u|^2 dx\nonumber\\
&\leq I_2+c_l\|(\nabla\Phi,\varrho,u,w,b)\|_{H^3}(\|\nabla^l \varrho\|_{L^2}^2+\|\nabla^{l+1}(\nabla\Phi,u,w,b)\|_{L^2}^2).
\end{align}

Next, for any integer $k\geq 0$, by the $\nabla^l (l=k, k+1, k+2)$ emergy estimate, for \eqref{1.3}$_3$-\eqref{1.3}$_4$ on $w$ and $b$, there holds
\begin{align}\label{3.35}
\frac{1}{2}&\frac{d}{dt}\int (|\nabla^l w|^2+|\nabla^l b|^2)dx+2\|\nabla^l w\|_{L^2}^2+3\|\nabla^{l+1}w\|_{L^2}^2+\|\nabla^{l+1}b\|_{L^2}^2\nonumber\\
=&\int\nabla^{l}(\nabla\times u)\cdot \nabla^l w dx-\int\nabla^l(u\cdot\nabla w)\cdot\nabla^l wdx
-\int\nabla^{l}(f(\varrho)\Delta w)\cdot\nabla^lwdx\nonumber\\
&-2\int\nabla^l(f(\varrho)\nabla\textrm{div}~w)\cdot\nabla^lwdx+2\int \nabla^l(f(\varrho)w)\cdot\nabla^lwdx\nonumber\\
&-\int\nabla^l(f(\varrho)\nabla\times u)\cdot\nabla^lwdx
+\int\nabla^l(b\cdot \nabla u)\cdot \nabla^l b dx-\int\nabla^l(u\cdot \nabla b)\cdot \nabla^lbdx\nonumber\\
&-\int\nabla^l(b\nabla\cdot u)\cdot \nabla^l b dx\nonumber\\
:=&\sum_{i=12}^{20}I_i.
\end{align}
Now, we proceed to estimate $I_{12}-I_{20}$. First, taking into account \eqref{3.11}-\eqref{3.14}, the terms $I_{13}, I_{18}, I_{19}, I_{20}$ can be estimated as
\begin{equation}\label{3.36}
I_{13}+I_{18}+I_{19}+I_{20}\leq c_l \|(u,w,b)\|_{H^3}\|\nabla^{l+1}(u,w,b)\|_{L^2}^2.
\end{equation}
Next, similar to the estimate of $I_4$ and $I_6$, we obtain
\begin{equation}\label{3.37}
I_{14}+I_{15} \leq c_l\|(\varrho,w)\|_{H^3}(\|\nabla^l \varrho\|_{L^2}^2+\|\nabla^{l+1} w \|_{L^2}^2),
\end{equation}
and
\begin{equation}\label{3.38}
I_{17}\leq c_l \|(\varrho,u)\|_{H^3}(\|\nabla^l \varrho\|_{L^2}^2+\|\nabla^{l+1}(u,w)\|_{L^2}^2).
\end{equation}
Finally, it is obvious that the term $I_{16}$ can be estimated as
\begin{equation}\label{3.39}
I_{16}\leq c_l \|(\varrho,w)\|_{H^3}\|\nabla^{l}(\varrho,w)\|_{L^2}^2.
\end{equation}
Inserting \eqref{3.36}-\eqref{3.39} into \eqref{3.35}, we obtain
\begin{align}\label{3.40}
\frac{1}{2}&\frac{d}{dt}\int(|\nabla^l w|^2+|\nabla^l b|^2)dx+2\|\nabla^lw\|_{L^2}^2
+3\|\nabla^{l+1}w\|_{L^2}^2+\|\nabla^{l+1}b\|_{L^2}^2\nonumber\\
\leq& I_{12}+c_l\|(\varrho,u,w,b)\|_{H^3}(\|\nabla^l \varrho\|_{L^2}^2+\|\nabla^{l+1}(u,w,b)\|_{L^2}^2)
+c_l\|(\varrho,w)\|_{H^3}\|\nabla^l(\varrho,w)\|_{L^2}^2.
\end{align}
Observe that, by the  Young's inequality
\begin{equation}\label{3.040}
I_{2}+I_{12}=2I_{12}\leq \frac43\|\nabla^{l+1}u\|_{L^2}^2+\frac32\|\nabla^lw\|_{L^2}^2,
\end{equation}
where in the last inequality, we have used fact $|\nabla\times u|^2\leq 2|\nabla u|^2$.

Combining \eqref{3.34} and \eqref{3.40}-(\ref{3.040}), yields that
\begin{align*}
&\frac{1}{2}\frac{d}{dt}\sum_{l=k}^{k+2}\|\nabla^l(\nabla\Phi,\varrho,u,w,b)\|_{L^2}^{2}+\frac12\sum_{l=k}^{k+2}\|\nabla^lw\|_{L^2}^{2}
+\sum_{l=k}^{k+2}(\frac23\|\nabla^{l+1}u\|_{L^2}^{2}+3\|\nabla^{l+1}w\|_{L^2}^{2}+\|\nabla^{l+1}b\|_{L^2}^{2})\nonumber\\
&\leq c_l\|(\nabla\Phi,\rho,u,w,b)\|_{H^3}\sum_{l=k}^{k+2}\left(\|\nabla^l\varrho\|_{L^2}^{2}+\|\nabla^{l+1}(\nabla\Phi,u,w,b)\|_{L^2}^{2}\right)
+c_l\|(\rho,w)\|_{H^3}\sum_{l=k}^{k+2}\|\nabla^lw\|_{L^2}^{2}.
\end{align*}
This, together with (\ref{3.1}),  whence \eqref{3.2}. Thus, we have completed the proof of Lemma \ref{le3.1}.
\end{proof}
Note that in Lemma \ref{le3.1}, we only derive the dissipation estimate of $u, w, b$. Now, we proceed to derive the dissipation estimate of $\nabla\Phi$
and $\varrho$ by constructing some interactive energy functions in the following lemma.
\begin{Lemma}\label{le3.2}
Let  all assumptions in Lemma \ref{le3.1} are in force, then for any $k \geq 0$, there holds
\begin{align}\label{3.41}
&\frac{d}{dt}\sum_{l=k}^{k+1}\int \nabla^l u\cdot \nabla^{l+1}\varrho dx
+\frac{1}{4}\sum_{l=k}^{k+1}(\|\nabla^{l}\varrho\|_{L^2}^2+\|\nabla^{l+1}\varrho\|_{L^2}^2+\|\nabla^{l+1}\nabla\Phi\|_{L^2}^2+
\|\nabla^{l+2}\nabla\Phi\|_{L^2}^2)\nonumber\\
\leq&c_l\delta\sum_{l=k}^{k+1}(\|\nabla^{l+1}\varrho\|_{L^2}^2+\|\nabla^{l+2}(w,b)\|_{L^2}^2)
+\sum_{l=k}^{k+1}(\|\nabla^{l+1}u\|_{L^2}^2+2\|\nabla^{l+1}w\|_{L^2}^2+4\|\nabla^{l+2}u\|_{L^2}^2).
\end{align}
\end{Lemma}
\begin{proof}
Applying $\nabla^l(l=k,k+1)$ to \eqref{1.3}$_2$ and then taking the $L^2-$ inner product with $\nabla^{l+1}\varrho$, we have
\begin{align*}
&\int \nabla^l \partial_tu\cdot\nabla^{l+1}\varrho dx+\|\nabla^{l+1}\varrho\|_{L^2}^2\\
=&\int\nabla^l\Delta u\cdot\nabla^{l+1}\varrho dx+\int\nabla^l\nabla\textrm{div}~u\cdot\nabla^{l+1}\varrho dx
+\int\nabla^l\nabla\times w\cdot\nabla^{l+1}\varrho dx\\
&-\int\nabla^l(u\cdot \nabla u)\cdot\nabla^{l+1}\varrho dx-\int\nabla^l(f(\varrho)\Delta u)\nabla^{l+1}\varrho dx
-\int \nabla^{l}(f(\varrho)\nabla\textrm{div}~u)\cdot\nabla^{l+1}\varrho dx\\
&-\int \nabla^l (f(\varrho)\nabla\times w)\cdot \nabla^{l+1}\varrho dx-\int\nabla^l(h(\varrho)\nabla\varrho)\cdot\nabla^{l+1}\varrho dx
+\int \nabla^{l}(g(\varrho)b\cdot \nabla b)\cdot\nabla^{l+1}\varrho dx\\
&-\frac{1}{2}\int \nabla^l (g(\varrho)\nabla|b|^2)\cdot\nabla^{l+1}\varrho dx+\int\nabla^l\nabla\Phi\cdot\nabla^l\nabla\varrho dx\\
:=&\sum_{j=1}^{11}J_i.
\end{align*}
On the other hand, it is clearly that
\begin{align}\label{3.42}
\sum_{i=1}^{11}J_{i}&=\frac{d}{dt}\int \nabla^l u\cdot\nabla^{l+1}\varrho dx-\int\nabla^l u\cdot\partial_{t}\nabla^{l+1}\varrho dx
+\|\nabla^{l+1}\varrho\|_{L^2}^2\nonumber\\
&=\frac{d}{dt}\int \nabla^l u\cdot\nabla^{l+1}\varrho dx+\int \nabla^{l+1}\textrm{div}~ u\cdot \nabla^{l}u dx
+\int\nabla^{l+1}\textrm{div}~(\varrho u)\cdot\nabla^{l}u dx+\|\nabla^{l+1}\varrho\|_{L^2}^2\nonumber\\
&:=\frac{d}{dt}\int\nabla^l u\cdot\nabla^{l+1}\varrho dx-\|\nabla^{l+1}u\|_{L^2}^2+\|\nabla^{l+1}\varrho\|_{L^2}^2+J_{12}.
\end{align}
Now, we concentrate our attention on estimating the terms $J_1-J_{12}$. First, employing the Cauchy's inequality, it holds that
\begin{align*}
J_1+J_2&\leq 4\|\nabla^{l+2} u\|_{L^2}^2+\frac{1}{4}\|\nabla^{l+1}\varrho\|_{L^2}^2,\\
J_3&\leq 2\|\nabla^{l+1}w\|_{L^2}^2+\frac{1}{4}\|\nabla^{l+1}\varrho\|_{L^2}^2,
\end{align*}
where in the last inequality, we have used the fact that $|\nabla\times w|^2\leq 2|\nabla w|^2$.

Moreover, taking into account \eqref{3.11}-\eqref{3.14},  we are in a position to obtain
\begin{align*}
J_4&=\int \textrm{div}\nabla^l(u\cdot \nabla u)\cdot \nabla^{l}\varrho dx\\
&\leq \|\textrm{div}\nabla^l(u\cdot \nabla u)\|_{L^{\frac{6}{5}}}\|\nabla^l \varrho\|_{L^{6}}\\
&\leq c_l\| u \|_{H^3}(\|\nabla^{l+1}\varrho\|_{L^2}^2+\|\nabla^{l+2}u\|_{L^2}^2).
\end{align*}
Similar to \eqref{3.18}, using the H\"{o}lder's inequality, Lemma \ref{le2.1} and Lemma \ref{le2.3}, the terms $J_5, J_6$ can be estimated as
\begin{align*}
J_5&\lesssim \|\sum_{s=0}^{l}C_l^s\nabla^s f(\varrho)\nabla^{l-s}\Delta u\|_{L^2}\|\nabla^{l+1}\varrho\|_{L^2}\\
&\leq c_l\|(\varrho,u)\|_{H^3}(\|\nabla^{l+1} \varrho\|_{L^2}^2+\|\nabla^{l+2}u\|_{L^2}^2),\\
J_6&\lesssim \|\sum_{s=0}^{l}C_l^s\nabla^s f(\varrho)\nabla^{l-s}\nabla\textrm{div}~u\|_{L^2}\|\nabla^{l+1}\varrho\|_{L^2}\\
&\leq c_l \|(\varrho,u)\|_{H^3}(\|\nabla^{l+1}\varrho\|_{L^2}^2+\|\nabla^{l+2}u\|_{L^2}^2).
\end{align*}
Furthermore, taking into account (\ref{3.30})-(\ref{3.31}), applying the H\"{o}lder's inequality,  Lemma \ref{le2.1} and Lemma \ref{le2.3}, we obtain
\begin{align*}
J_7&\lesssim \|\sum_{s=0}^{l}(\nabla^s f(\varrho)\nabla^{l-s}\nabla\times w)\|_{L^2}\|\nabla^{l+1}\varrho\|_{L^2}\\
&\leq c_l \|(\varrho,w)\|_{H^3}(\|\nabla^{l+1}\varrho\|_{L^{2}}^2+\|\nabla^{l+2}w\|_{L^2}^2),\\
J_8&\lesssim \|\sum_{s=0}^{l}(\nabla^s h(\varrho)\nabla^{l-s}\nabla \varrho)\|_{L^2}\|\nabla^{l+1}\varrho\|_{L^2}\\
&\leq c_l\|\varrho\|_{H^3}\|\nabla^{l+1}\varrho\|_{L^2}^2.
\end{align*}
Similar to the estimates of  \eqref{3.22}-\eqref{3.33}, we  further obtain
\begin{align*}
J_9+J_{10}&\lesssim (\|\sum_{s=0}^{l}\nabla^s g(\varrho)\nabla^{l-s}\nabla\cdot(b\otimes b)\|_{L^2}
+\|\sum_{s=0}^{l}\nabla^s g(\varrho)\nabla^{l-s}\nabla|b|^2\|_{L^2})\|\nabla^{l+1}\varrho\|_{L^2}^2\\
&\leq \|(\varrho,b)\|_{H^3}(\|\nabla^{l+1}\varrho\|_{L^2}^2+\|\nabla^{l+2}b\|_{L^2}^2).
\end{align*}
Next, by  integration by parts and (\ref{1.3})$_5$, we infer that
\begin{equation*}
J_{11}=-\int\nabla^l\Delta\Phi\cdot\nabla^l\varrho dx=-\int|\nabla^l\varrho|^2dx.
\end{equation*}
Furthermore, from (\ref{1.3})$_5$, there holds
\begin{equation}\label{3.042}
\|\nabla^{l+1}\nabla\Phi\|_{L^2}^2=\|\nabla^{l}\Delta\Phi\|_{L^2}^2=\|\nabla^{l}\varrho\|_{L^2}^2,\quad \|\nabla^{l+2}\nabla\Phi\|_{L^2}^2=\|\nabla^{l+1}\varrho\|_{L^2}^2.
\end{equation}
Finally, by integration by parts and Lemma \ref{le2.1}, we get
\begin{align*}
-J_{12}&=\int \nabla^{l+1}(\varrho u)\cdot \nabla^{l+1}u dx\\
&=\int(\nabla^{l+1}\varrho u)\cdot\nabla^{l+1}u dx+\int \varrho \nabla^{l+1}u\cdot\nabla^{l+1}u dx
+\int\sum_{s=1}^lC_{l+1}^s(\nabla^s \varrho \nabla^{l+1-s}u)\cdot\nabla^{l+1}u dx\\
&\lesssim(\| u \|_{\infty}+\| \varrho\|_{\infty})(\|\nabla^{l+1}\varrho\|_{L^2}^2+\|\nabla^{l+1}u\|_{L^2}^2)
+\sum_{s=1}^l\|\nabla^s \varrho\nabla^{l+1-s}u\|_{L^2}\|\nabla^{l+1}u\|_{L^2}\\
&\leq c_l\|(\varrho,u)\|_{H^3}(\|\nabla^{l+1}\varrho\|_{L^2}^2+\|\nabla^{l+1}u\|_{L^2}^2).
\end{align*}
Putting these estimations into \eqref{3.42}, and summing up with $l=k,k+1$, we finally obtain
\begin{align*}
&\frac{d}{dt}\sum_{l=k}^{k+1}\int \nabla^l u\cdot \nabla^{l+1}\varrho dx
+\frac{1}{4}\sum_{l=k}^{k+1}(\|\nabla^{l}\varrho\|_{L^2}^2+\|\nabla^{l+1}\varrho\|_{L^2}^2+\|\nabla^{l+1}\nabla\Phi\|_{L^2}^2+
\|\nabla^{l+2}\nabla\Phi\|_{L^2}^2)\nonumber\\
\leq&c_l\|(\varrho,u,w,b)\|_{H^3}\sum_{l=k}^{k+1}\|\nabla^{l+1}\varrho\|_{L^2}^2
+(1+c_l\|(\varrho,u)\|_{H^3})\sum_{l=k}^{k+1}\|\nabla^{l+1}u\|_{L^2}^2\nonumber\\
&+(4+c_l\|(\varrho,u)\|_{H^3})\sum_{l=k}^{k+1}\|\nabla^{l+2}u\|_{L^2}^2+2\sum_{l=k}^{k+1}\|\nabla^{l+1}w\|_{L^2}^2
+c_l\|(\varrho,w,b)\|_{H^3}\sum_{l=k}^{k+1}\|\nabla^{l+2}(w,b)\|_{L^2}^2.
\end{align*}
This, together with (\ref{3.1}) implies \eqref{3.41}. Thus, we have completed the proof of Lemma \ref{le3.2}.
\end{proof}
\subsection{Local existence of solution}
In this subsection, we devote to proving the local existence of solution $(\nabla\Phi,\varrho,u,w,b)$ in $H^3$-norm. Firstly, we construct the solution
sequence $(\nabla\Phi^j,\varrho^j,u^j,w^j,b^j)_{j\geq0}$ by solving iteratively the following Cauchy problem for $j\geq 0$:
\begin{align}\label{3.143}
\begin{cases}
\partial_t\varrho^{j+1}+\textrm{div}~u^{j+1}=M_1^{j+1},& \\
\partial_tu^{j+1}+\nabla\varrho^{j+1}-\Delta u^{j+1}-\nabla \textrm{div}~u^{j+1}
- \nabla\times w^{j+1}-\nabla\Phi^{j+1}=M_2^{j+1},&\\
\partial_tw^{j+1}+2 w^{j+1}-\Delta w^{j+1}-2\nabla \textrm{div}~w^{j+1}-\nabla \times u^{j+1}=M_3^{j+1},&\\
\partial_tb^{j+1}-\Delta b^{j+1}= M_4^{j+1},&\\
\Delta\Phi^{j+1}=\varrho^{j+1},&\\
\textrm{div}~b^{j+1}=0, \quad t>0, x\in\mathbb{R}^3,
\end{cases}
\end{align}
with initial data
\begin{equation}\label{3.144}
(\varrho^{j+1},u^{j+1},w^{j+1},b^{j+1})(x,0)=(\varrho_0,u_0,w_0,b_0)\longrightarrow(0,0,0,0) \quad \text{as}\  |x|\longrightarrow \infty,
\end{equation}
where the nonlinear terms $M_i^{j+1}$ $(i=1,2,3,4)$ are defined as
\begin{align*}
M_1^{j+1}=&-(\varrho^{j}\textrm{div}~ u^{j+1}+u^j\cdot\nabla \varrho^{j+1}),\\
M_2^{j+1}=&-u^j\cdot \nabla u^{j+1}-f(\varrho^j)[\Delta u^{j+1}+\nabla \textrm{div}~u^{j+1}+ \nabla \times w^{j+1}]\\
&-h(\varrho^j)\nabla\varrho^{j+1}+g(\varrho^j)[(\nabla\times b^{j+1})\times b^j],\\
M_3^{j+1}=&-u^j\cdot \nabla w^{j+1}-f(\varrho^j)[\Delta w^{j+1}+2\nabla \textrm{div}~w^{j+1}-2 w^{j+1}+\nabla \times u^{j+1}],\\
M_4^{j+1}=&b^j\cdot \nabla u^{j+1}-u^j\cdot \nabla b^{j+1}-b^j\textrm{div}~u^{j+1},
\end{align*}
where $(\nabla\Phi^0,\varrho^0,u^0,w^0,b^0)\equiv (0,0,0,0,0)$ is set at initial step. In what follows, for simplicity, we may denote
$(\nabla\Phi^j,\varrho^j,$ $u^j,w^j,b^j)_{j\geq 0}$ and $(\nabla\Phi(0),\varrho_0,u_0,w_0,b_0)$ by $(\mathscr{A}^j)_{j\geq 0}$ and
$\mathscr{A}_0$, respectively. Then, one has the following result.
\begin{Lemma}\label{le3.3}
Let  all assumptions in Theorem \ref{th1.1} hold for $N=3$. Then, there holds
\begin{align}\label{3.043}
\frac{1}{2} &\frac{d}{dt}\|\mathscr{A}^{j+1}(t)\|_{H^3}^2+\|(\sqrt{\frac23}\nabla u^{j+1},\sqrt{3}\nabla w^{j+1},\nabla b^{j+1} )\|_{H^3}^2
+\frac12\|w^{j+1}\|_{H^3}^2 \nonumber\\
\leq&c \|(\varrho^j,u^j,b^j)\|_{H^3}\|(\nabla^2\Phi^{j+1},\varrho^{j+1},\nabla u^{j+1},\nabla w^{j+1},\nabla b^{j+1})\|_{H^3}^2
+c\|\varrho^j\|_{H^3}\|w^{j+1}\|_{H^3}^2\nonumber\\
&+c\|\varrho^j\|_{H^3}\|b^j\|_{H^3}\|\nabla(u^{j+1},b^{j+1})\|_{H^3}^2+c\|(\varrho^j,u^j)\|_{H^3}\|\nabla \Phi^{j+1}\|_{H^3}^2,
\end{align}
where $c$ is a positive constant independent of $j$.
\end{Lemma}
\begin{proof}
Similar to Lemma \ref{le3.1}, for $0\leq l\leq N$, $N=3$, for  (\ref{3.143})$_1$-(\ref{3.143})$_2$ on $\varrho^{j+1}$ and $u^{j+1}$, we get
\begin{align}\label{3.46}
\frac12&\frac{d}{dt} \int|\nabla^l \varrho^{j+1}|^2+|\nabla^l u^{j+1}|^2 dx+2\int|\nabla^{l+1} u^{j+1}|^2 dx
-\int\nabla^l\nabla\Phi^{j+1}\cdot \nabla^l u^{j+1}dx\nonumber\\
=&\int-\nabla^l(u^{j}\cdot\nabla \varrho^{j+1}) \cdot\nabla^l\varrho^{j+1}-\nabla^l(\varrho^j{\rm div}~u^{j+1}) \cdot\nabla^l\varrho^{j+1}
+\nabla^l(\nabla\times w^{j+1})\cdot \nabla^l u^{j+1}\nonumber\\
&-\nabla^l(u^j\cdot\nabla u^{j+1})\cdot\nabla^l u^{j+1}-\nabla^l(f(\varrho^j)\Delta u^{j+1})\cdot \nabla^l u^{j+1}
-\nabla^l(f(\varrho^j)\nabla \textrm{div}~u^{j+1})\cdot \nabla^l u^{j+1}\nonumber\\
&-\nabla^l(f(\varrho^j)\nabla\times w^{j+1})\cdot \nabla^l u^{j+1}-\nabla^l(h(\varrho^j)\nabla \varrho^{j+1})\cdot \nabla^l u^{j+1}
+\nabla^l[g(\varrho^j)(\nabla\times b^{j+1})\times b^j]\cdot \nabla^l u^{j+1}dx\nonumber\\
:=&\sum_{i=1}^9H_i.
\end{align}
First of all, by integration by parts and (\ref{3.143})$_5$, we infer that
\begin{align}\label{3.046}
-\int\nabla^l\nabla\Phi^{j+1}\cdot\nabla^l u^{j+1}dx&= \int\nabla^l\Phi^{j+1}\cdot\nabla^l\textrm{div}~ u^{j+1}dx \nonumber\\
&=-\int\nabla^l\Phi^{j+1}\cdot\nabla^l\partial_t\varrho^{j+1}
+\nabla^l\Phi^{j+1}\cdot\nabla^l(\varrho^{j}\textrm{div}~ u^{j+1}+u^j\cdot\nabla \varrho^{j+1})dx \nonumber\\
&=-\int\nabla^l\Phi^{j+1}\cdot\nabla^l\partial_t\Delta\Phi^{j+1}
+\nabla^l\Phi^{j+1}\cdot\nabla^l(\varrho^{j}\textrm{div}~ u^{j+1}+u^j\cdot\nabla \varrho^{j+1})dx \nonumber\\
&=\frac12\frac d{dt}\int|\nabla^l\nabla\Phi^{j+1}|^2dx
-\int\nabla^l\Phi^{j+1}\cdot\nabla^l(\varrho^{j}\textrm{div}~ u^{j+1}+u^j\cdot\nabla \varrho^{j+1})dx.
\end{align}
Applying Lemma \ref{le2.1} and the Young's inequality, it is easy to see that
\begin{align}\label{3.146}
\int\nabla^l\Phi^{j+1}\cdot\nabla^l(\varrho^{j}\textrm{div}~ u^{j+1})dx&\leq\|\nabla^l\Phi^{j+1}\|_{L^6}
\|\nabla^l(\varrho^{j}\textrm{div}~ u^{j+1})\|_{L^{\frac65}}\nonumber\\
&\leq c\|\varrho^j\|_{H^3}(\|\nabla\Phi^{j+1}\|_{H^3}^2+\|\nabla u^{j+1}\|_{H^3}^2).
\end{align}
On the other hand, if $l=0$, it holds that
\begin{align}\label{3.461}
\int\Phi^{j+1}\cdot(u^j\cdot\nabla \varrho^{j+1})dx& =-\int\nabla\Phi^{j+1}\cdot(u^j\varrho^{j+1})
+\Phi^{j+1}(\varrho^{j+1}\textrm{div}~ u^{j})dx\nonumber \\
& \leq \|\nabla \Phi^{j+1}\|_{L^2}\|u^j\|_{L^{\infty}}\|\varrho^{j+1}\|_{L^2}+\|\Phi^{j+1}\|_{L^6}\|\textrm{div}~u^j\|_{L^3}\|\varrho^{j+1}\|_{L^2}
\nonumber \\
& \leq c\|u^j\|_{H^3}(\|\nabla \Phi^{j+1}\|^2_{L^2}+\|\varrho^{j+1}\|^2_{L^2}).
\end{align}
If $l\geq1$, then we obtain
\begin{align}\label{3.058}
\int\nabla^l\Phi^{j+1}\cdot\nabla^l(u^j\cdot\nabla \varrho^{j+1})dx&=-\int(\nabla\cdot\nabla^l\Phi^{j+1})
\cdot\nabla^{l-1}(u^j\cdot\nabla \varrho^{j+1})dx\nonumber \\
&\leq \|\nabla\cdot\nabla^l\Phi^{j+1}\|_{L^2}\|\nabla^{l-1}(u^j\cdot\nabla \varrho^{j+1})\|_{L^2}\nonumber \\
&\leq c\|u^j\|_{H^3}(\|\nabla \Phi^{j+1}\|^2_{H^3}+\|\varrho^{j+1}\|^2_{H^3}).
\end{align}
Inserting (\ref{3.146})-(\ref{3.058}) into (\ref{3.046}), we have
\begin{align}\label{3.047}
-\int\nabla^l\nabla\Phi^{j+1}\cdot\nabla^l u^{j+1}dx\geq&\frac12\frac d{dt}\int|\nabla^l\nabla\Phi^{j+1}|^2dx
-c_l\|(\varrho^j,u^j)\|_{H^3}\|\nabla\Phi^{j+1}\|^2_{H^3}\nonumber\\
&-c_l\|(\varrho^j,u^j)\|_{H^3}\|(\varrho^{j+1},\nabla u^{j+1})\|^2_{H^3}.
\end{align}
Now, we focus our attention on estimating the terms $H_1$-$H_9$. Note that, it is trivial for $l=0$, thus, in what follows, we may only consider the case
that $1\leq l\leq 3$. For the term $H_1$, by integration by parts and the H\"{o}lder's inequality, we can see that
\begin{align}\label{3.47}
H_1=&-\int\sum_{s=1}^lC_l^s(\nabla^su^j\cdot \nabla^{l-s}\nabla\varrho^{j+1})\cdot\nabla^l \varrho^{j+1}dx
-\int( u^j\cdot \nabla\nabla^l\varrho^{j+1})\cdot\nabla^l\varrho^{j+1}dx\nonumber \\
\leq &c\|u^j\|_{H^3}\|\varrho^{j+1}\|^2_{H^3}.
\end{align}
Likewise, we can also obtain
\begin{align}
&H_2\leq c\|\varrho^j\|_{H^3}\|(\varrho^{j+1},\nabla u^{j+1})\|^2_{H^3},\label{3.48}  \\
&H_4\leq c\|u^j\|_{H^3}\|\nabla u^{j+1}\|^2_{H^3}.\label{3.49}
\end{align}
Next, for the term $H_5$, by the H\"{o}lder's inequality and Lemma \ref{le2.3}, it holds that
\begin{align}\label{3.50}
H_5&=\int \nabla^{l-1} (f(\varrho^j)\Delta u^{j+1})\cdot\textrm{div}\nabla^lu^{j+1}\nonumber \\
& \lesssim \sum_{s=0}^{l-1}\|\nabla^sf(\varrho^j)\cdot \nabla^{l-1-s}\Delta u^{j+1}\|_{L^2}\|\nabla^{l+1}u^{j+1}\|_{L^2}\nonumber \\
&\leq c\|\varrho^j\|_{H^3}\|\nabla u^{j+1}\|^2_{H^3}.
\end{align}
Similarly, we further obtain
\begin{align}
 & H_6\leq c\|\varrho^j\|_{H^3}\|\nabla u^{j+1}\|^2_{H^3}\label{3.51}\\
& H_7\leq c\|\varrho^j\|_{H^3}\|\nabla(w^{j+1},u^{j+1})\|^2_{H^3}\label{3.52}\\
 & H_8\leq c\|\varrho^j\|_{H^3}\|(\varrho^{j+1},\nabla u^{j+1})\|^2_{H^3}.\label{3.53}
\end{align}
Finally, for the term $H_9$, note that $g(\varrho)=1-f(\varrho)$, using the H\"{o}lder and Young's inequality, we deduce that
\begin{align}\label{3.54}
H_9=& \int g(\varrho^j)\nabla^l((\nabla\times b^{j+1})\times b^j)\cdot \nabla^lu^{j+1}dx
-\sum_{s=1}^lC_l^s\int[\nabla^sf(\varrho^j)\nabla^{l-s}((\nabla\times b^{j+1})\times b^j)]\cdot \nabla^lu^{j+1}dx\nonumber \\
\leq& \|\nabla^l((\nabla\times b^{j+1})\times b^j)\|_{L^{\frac65}}\|\nabla^lu^{j+1}\|_{L^6}
+\sum_{s=1}^{[\frac l2]}\|\nabla^s f(\varrho^j)\|_{L^3}\|\nabla^{l-s}((\nabla\times b^{j+1})\times b^j)\|_{L^2}\|\nabla^lu^{j+1}\|_{L^6}\nonumber \\
&+\sum_{s=[\frac l2]+1}^{l}\|\nabla^s f(\varrho^j)\|_{L^2}\|\nabla^{l-s}((\nabla\times b^{j+1})\times b^j)\|_{L^3}\|\nabla^lu^{j+1}\|_{L^6}\nonumber \\
\leq& c\|b^j\|_{H^3}\|\nabla( u^{j+1}, b^{j+1})\|^2_{H^3}+ c\|\varrho^j\|_{H^3}\|b^j\|_{H^3}\|\nabla( u^{j+1}, b^{j+1})\|^2_{H^3}.
\end{align}
Plugging (\ref{3.047})-(\ref{3.54}) into (\ref{3.46}), we conclude that
\begin{align}\label{3.55}
 \frac12&\frac d{dt}\int|\nabla^l\varrho^{j+1}|^2+|\nabla^lu^{j+1}|^2+|\nabla^l\nabla\Phi^{j+1}|^2dx+2\int|\nabla^{l+1}u^{j+1}|^2dx\nonumber\\
\leq& H_3+c\|(\varrho^j,u^j,b^j)\|_{H^3}\|(\nabla^2\Phi^{j+1},\varrho^{j+1},\nabla u^{j+1},\nabla w^{j+1},\nabla b^{j+1})\|^2_{H^3}\nonumber\\
&+c\|\varrho^j\|_{H^3}\|b^j\|_{H^3}\|\nabla( u^{j+1}, b^{j+1})\|^2_{H^3}+c\|(\varrho^j,u^j)\|_{H^3}\|\nabla\Phi^{j+1}\|^2_{H^3}.
\end{align}
On the other hand, from (\ref{3.143})$_3$-(\ref{3.143})$_4$, we can see that for $1\leq l\leq3$
\begin{align}\label{3.56}
\frac12&\frac d{dt}\int |\nabla^lw^{j+1}|^2+|\nabla^lb^{j+1}|^2dx
+\|\nabla^{l+1}(\sqrt{3}w^{j+1},b^{j+1})\|^2_{L^2}+2\|\nabla^lw^{j+1}\|^2_{L^2}\nonumber \\
=&\int\nabla^l(\nabla\times u^{j+1})\cdot\nabla^l w^{j+1}-\nabla^l(u^j\cdot\nabla w^{j+1})\cdot \nabla^lw^{j+1}
-\nabla^l(f(\varrho^j)\Delta w^{j+1})\cdot\nabla^lw^{j+1}\nonumber \\
&-2\nabla^l(f(\varrho^j)\nabla\textrm{div}w^{j+1})\cdot\nabla^lw^{j+1}+2\nabla^l(f(\varrho^j)w^{j+1})\cdot\nabla^lw^{j+1}
-\nabla^l(f(\varrho^j)\nabla\times u^{j+1})\cdot\nabla^lw^{j+1}\nonumber \\
&+\nabla^l(b^j\cdot\nabla u^{j+1})\cdot \nabla^l b^{j+1}-\nabla^l(u^j\cdot\nabla b^{j+1})\cdot \nabla^l b^{j+1}
-\nabla^l(b^j\textrm{div}~u^{j+1})\cdot \nabla^l b^{j+1}dx\nonumber \\
:=&\sum_{i=10}^{18}H_i.
\end{align}
Similar to (\ref{3.47})-(\ref{3.49}), we arrive at
\begin{align}
&H_{11}\leq c\|u^j\|_{H^3}\|\nabla w^{j+1}\|^2_{H^3},\label{3.57} \\
&H_{17}\leq c\|u^j\|_{H^3}\|\nabla b^{j+1}\|^2_{H^3},\label{3.58}
\end{align}
and
\begin{equation}\label{3.59}
H_{16}+H_{18}\leq c\|b^j\|_{H^3}\|\nabla (u^{j+1},b^{j+1})\|^2_{H^3}.
\end{equation}
Next, taking into account (\ref{3.50}), we are in a position to obtain
\begin{equation}\label{3.60}
H_{12}+H_{13}\leq c\|\varrho^j\|_{H^3}\|\nabla w^{j+1}\|^2_{H^3}.
\end{equation}
Furthermore, in virtue of (\ref{2.3}) and the H\"{o}lder's inequality, for the term $H_{14}$, we can find that
\begin{align}\label{3.61}
H_{14}&= 2\sum_{s=0}^{[\frac l2]}C^s_{l}\int\nabla^sf(\varrho^j)\nabla^{l-s}w^{j+1}\cdot \nabla^{L}w^{j+1} dx
+2\sum_{s=[\frac l2]+1}^{l}C^s_{l}\int\nabla^sf(\varrho^j)\nabla^{l-s}w^{j+1}\cdot \nabla^{L}w^{j+1} dx\nonumber \\
&\lesssim \sum_{s=0}^{[\frac l2]}\|\nabla^sf(\varrho^j)\|_{L^{\infty}}\|\nabla^{l-s}w^{j+1}\|_{L^{2}}\|\nabla^{l}w^{j+1}\|_{L^{2}}
+\sum_{s=[\frac l2]+1}^{l}\|\nabla^sf(\varrho^j)\|_{L^{2}}\|\nabla^{l-s}w^{j+1}\|_{L^{3}}\|\nabla^{l}w^{j+1}\|_{L^{6}}\nonumber \\
&\leq c \|\varrho^j\|_{H^3}(\|w^{j+1}\|^2_{H^3}+\|\nabla w^{j+1}\|^2_{H^3}).
\end{align}
Similar to (\ref{3.52}), the term $H_{15}$ can be estimated as
\begin{equation}\label{3.62}
H_{15}\leq c\|\varrho^j\|_{H^3}\|\nabla(u^{j+1}, w^{j+1})\|^2_{H^3}.
\end{equation}
Inserting (\ref{3.57})-(\ref{3.62}) into (\ref{3.56}), we conclude that
\begin{align}\label{3.63}
\frac12&\frac d{dt}\|\nabla^l(w^{j+1},b^{j+1}\|^2_{L^2}
+\|\nabla^{l+1}(\sqrt{3}w^{j+1},b^{j+1})\|^2_{L^2}+2\|\nabla^lw^{j+1}\|^2_{L^2}\nonumber \\
\leq&c\|(\varrho^j,u^j,b^j)\|_{H^3}\|\nabla(u^{j+1},w^{j+1},b^{j+1})\|^2_{H^3}+c\|\varrho^j\|_{H^3}\|w^{j+1}\|^2_{H^3}+H_{10}.
\end{align}
Now, combining (\ref{3.55})  and (\ref{3.63}), taking into account that $H_3+H_{10}=2H_{10}$ and (\ref{3.040}), then we have (\ref{3.043}).
This completes the proof of Lemma \ref{le3.3}.
\end{proof}
Similar to Lemma \ref{le3.2}, now, we are going to deduce the dissipative estimate of $\nabla\Phi^{j+1},\varrho^{j+1}$.
\begin{Lemma}\label{le3.4}
Let  all assumptions in Theorem \ref{th1.1} hold for $N=3$. Then, for $l=0,1,2$, there holds
\begin{align}\label{3.64}
\frac{d}{dt}&\sum_{l=0}^2\int\nabla^lu^{j+1}\cdot\nabla^{l+1}\varrho^{j+1}dx
+\frac{1}{4}\|(\varrho^{j+1},\nabla\varrho^{j+1},\nabla^2\Phi^{j+1},\nabla^3\Phi^{j+1})\|^2_{H^2} \nonumber\\
\leq&\|\nabla u^{j+1}\|^2_{H^2}+2\|\nabla w^{j+1}\|^2_{H^2}+4\|\nabla^2 u^{j+1}\|^2_{H^2}
+c\|\varrho^j\|_{H^3}\|b^j\|_{H^3}\|\nabla(\varrho^{j+1},b^{j+1})\|^2_{H^2}\nonumber\\
&+c\|(\varrho^j,u^j,b^j)\|_{H^3}(\|\nabla(\varrho^{j+1},u^{j+1},w^{j+1},b^{j+1})\|^2_{H^2}+\|\nabla^2u^{j+1}\|^2_{H^2}).
\end{align}
\end{Lemma}
\begin{proof}
Applying $\nabla^l$ $(0\leq l\leq2)$ to (\ref{3.143})$_2$ and then taking the $L^2$ inner product with $\nabla^{l+1}\varrho^{j+1}$, we obtain
\begin{align*}
&\int\nabla^l\partial_t u^{j+1}\cdot \nabla^{l+1}\varrho^{j+1}dx+\|\nabla^{l+1}\varrho^{j+1}\|^2_{L^2}
= \int \nabla^l\Delta u^{j+1}\cdot\nabla^{l+1}\varrho^{j+1}dx\\
&+\int \nabla^l\nabla\textrm{div}~ u^{j+1}\cdot\nabla^{l+1}\varrho^{j+1}dx+\int \nabla^l\nabla\times w^{j+1}\cdot\nabla^{l+1}\varrho^{j+1}dx
-\int \nabla^l(u^j\cdot\nabla u^{j+1})\cdot\nabla^{l+1}\varrho^{j+1}dx\\
&-\int\nabla^l(f(\varrho^j)\Delta u^{j+1})\cdot\nabla^{l+1}\varrho^{j+1}dx
-\int\nabla^l(f(\varrho^j)\nabla\textrm{div}~ u^{j+1})\cdot\nabla^{l+1}\varrho^{j+1}dx\\
&-\int\nabla^l(f(\varrho^j)\nabla\times w^{j+1})\cdot\nabla^{l+1}\varrho^{j+1}dx
-\int\nabla^l(h(\varrho^j)\nabla\varrho^{j+1})\cdot\nabla\varrho^{j+1}dx\\
&+\int\nabla^l(g(\varrho^j)(\nabla\times b^{j+1})\times b^j)\cdot\nabla^{l+1}\varrho^{j+1}dx
+\int \nabla^l\nabla\Phi^{j+1}\cdot\nabla^{l+1}\varrho^{j+1}dx\\
&:=\sum_{i=1}^{10}W_i.
\end{align*}
On the other hand, we can see that
\begin{align}\label{3.65}
\sum_{i=1}^{10}W_i=&\frac{d}{dt}\int\nabla^lu^{j+1}\cdot\nabla^{l+1}\varrho^{j+1}dx
- \int\nabla^lu^{j+1}\cdot\nabla^{l+1}\partial_t\varrho^{j+1}dx+\|\nabla^{l+1}\varrho^{j+1}\|^2_{L^2}\nonumber \\
=&\frac{d}{dt}\int\nabla^lu^{j+1}\cdot\nabla^{l+1}\varrho^{j+1}dx
+\int\nabla^lu^{j+1}\cdot\nabla^{l+1}\textrm{div}~u^{j+1}dx\nonumber \\
&+\int\nabla^lu^{j+1}\cdot\nabla^{l+1}(\varrho^j\textrm{div}~u^{j+1}+u^j\cdot\nabla \varrho^{j+1})dx+\|\nabla^{l+1}\varrho^{j+1}\|^2_{L^2}\nonumber \\
:=&\frac{d}{dt}\int\nabla^lu^{j+1}\cdot\nabla^{l+1}\varrho^{j+1}dx-\|\nabla^{l+1}u^{j+1}\|^2_{L^2}+\|\nabla^{l+1}\varrho^{j+1}\|^2_{L^2}
+W_{11}.
\end{align}
Now, we turn to estimate the terms $W_1-W_{11}$. First, by the Cauchy's inequality, we can see that $W_1$--$W_3$ can be estimated as
\begin{align*}
 W_1+W_2&\leq \frac14\|\nabla^{l+1}\varrho^{j+1}\|^2_{L^2}+4\|\nabla^{l+2}u^{j+1}\|^2_{L^2}  \\
 & \leq \frac14\|\nabla\varrho^{j+1}\|^2_{H^2}+4\|\nabla^{2}u^{j+1}\|^2_{H^2}\\
W_3&\leq \frac14\|\nabla^{l+1}\varrho^{j+1}\|^2_{L^2}+2\|\nabla^{l+1}w^{j+1}\|^2_{L^2}  \\
 & \leq \frac14\|\nabla\varrho^{j+1}\|^2_{H^2}+2\|\nabla^{2}w^{j+1}\|^2_{H^2}.
\end{align*}
Furthermore, applying the H\"{o}lder's inequality and Lemma \ref{le2.1}, for the term $W_4$, we have
\begin{equation*}
W_4\leq c\|u^j\|_{H^3}\|\nabla(\varrho^{j+1},u^{j+1})\|^2_{H^2}.
\end{equation*}
Similar to (\ref{3.50}), we obtain
\begin{equation*}
W_5+W_6\leq c\|\varrho^j\|_{H^3}(\|\nabla\varrho^{j+1}\|^2_{H^2}+\|\nabla^2u^{j+1}\|^2_{H^2}).
\end{equation*}
In addition, taking into account (\ref{3.52}), we deduce that
\begin{equation*}
W_7+W_8\leq c\|\varrho^j\|_{H^3}\|\nabla(\varrho^{j+1},w^{j+1})\|^2_{H^2}).
\end{equation*}
Next, similar to (\ref{3.54}), it holds that
\begin{align*}
W_9&=\int g(\varrho^j)\nabla^l((\nabla\times b^{j+1})\times b^j)\cdot \nabla^{l+1}\varrho^{j+1}dx
-\sum_{s=1}^lC_l^s\int[\nabla^s f(\varrho^j)\nabla^{l-s}((\nabla\times b^{j+1})\times b^j)]\cdot \nabla^{l+1}\varrho^{j+1}dx\\
&\leq c\|b^j\|_{H^3}\|\nabla(\varrho^{j+1},b^{j+1})\|^2_{H^2})+c\|\varrho^j\|_{H^3}\|b^j\|_{H^3}\|\nabla(\varrho^{j+1},b^{j+1})\|^2_{H^2}).
\end{align*}
Moreover, by (\ref{3.143})$_5$, it is obvious that
\begin{equation*}
W_{10}=-\int|\nabla^l\varrho^{j+1}|^2dx.
\end{equation*}
Finally, the term $W_{11}$ can be estimated as
\begin{align*}
W_{11}&=-\int(\nabla\cdot\nabla^lu^{j+1})\cdot\nabla^{l}(\varrho^j\textrm{div}~u^{j+1}+u^j\cdot\nabla\varrho^{j+1})dx  \\
 &\leq c\|(\varrho^j,u^j)\|_{H^3}\|(\nabla \varrho^{j+1},\nabla u^{j+1})\|^2_{H^2}.
\end{align*}
Putting these estimates into (\ref{3.65}) and summing up with $l=0,1,2$, we finally obtain (\ref{3.64}). This completes the proof of Lemma \ref{le3.4}.
\end{proof}
Based on Lemma \ref{le3.3} and \ref{le3.4},  we then immediately have the following result.
\begin{Theorem}\label{th3.1}
Let  all assumptions in Theorem \ref{th1.1} hold for $N=3$. There are constants $\varepsilon_1>0$, $T_1>0$, $M_0>0$ such that if
$\|\mathscr{A}_0\|_{H^3}\leq\varepsilon_1$, then for each $j\geq0$, $\mathscr{A}^j\in C([0,T_1];H^3)$ is well defined and
\begin{equation}\label{3.065}
\sup_{0\leq t\leq T_1}\|\mathscr{A}^j(t)\|_{H^3}\leq M_0,\quad j\geq0.
\end{equation}
Furthermore, $(\mathscr{A}^j)_{j\geq0}$ is a Cauchy sequence in the Banach space $C([0,T_1];H^3)$, the corresponding limit function denoted by
$\mathscr{A}(t)$ belongs to $C([0,T_1];H^3)$ with
\begin{equation}\label{3.66}
\sup_{0\leq t\leq T_1}\|\mathscr{A}(t)\|_{H^3}\leq M_0,
\end{equation}
and $\mathscr{A}=(\nabla\Phi,\varrho,u,w,b)$ is a solution over $[0,T_1]$ to the Cauchy problem (\ref{1.3})-(\ref{1.4}). Finally, the Cauchy problem
(\ref{1.3})-(\ref{1.4}) admits at most one solution in $C([0,T_1];H^3)$ satisfying (\ref{3.66}).
\end{Theorem}
\begin{proof}
First of all, we shall prove (\ref{3.065}) by induction. The trivial case is $j=0$ since $\mathscr{A}^0=0$ by the assumption at initial step. Suppose that
it is true for $j\geq0$ with $M_0>0$ small enough to be determined later. Now, we propose to prove it for $j+1$, we need some energy estimates on
$\mathscr{A}^{j+1}$. By the induction assumptions and Lemma \ref{le3.3}-\ref{le3.4}, we have
\begin{align}\label{3.67}
\frac{1}{2}& \frac{d}{dt}\|\mathscr{A}^{j+1}(t)\|_{H^3}^2+\|(\sqrt{\frac23}\nabla u^{j+1},\sqrt{3}\nabla w^{j+1},\nabla b^{j+1} )\|_{H^3}^2
+\frac12\|w^{j+1}\|_{H^3}^2 \nonumber\\
\leq& cM_0\|(\nabla^2\Phi^{j+1},\varrho^{j+1},\nabla u^{j+1},\nabla w^{j+1},\nabla b^{j+1})\|_{H^3}^2
+cM_0^2\|\nabla(u^{j+1},b^{j+1})\|_{H^3}^2+cM_0\|w^{j+1}\|_{H^3}^2\nonumber\\
&+cM_0\|\nabla\Phi^{j+1}\|_{H^3}^2,
\end{align}
and
\begin{align}\label{3.68}
\frac{d}{dt}&\sum_{l=0}^2\int\nabla^lu^{j+1}\cdot\nabla^{l+1}\varrho^{j+1}dx
+\frac{1}{4}\|(\varrho^{j+1},\nabla\varrho^{j+1},\nabla^2\Phi^{j+1},\nabla^3\Phi^{j+1})\|^2_{H^2} \nonumber\\
\leq&\|\nabla u^{j+1}\|^2_{H^2}+2\|\nabla w^{j+1}\|^2_{H^2}+4\|\nabla^2 u^{j+1}\|^2_{H^2}
+cM_0^2\|\nabla(\varrho^{j+1},b^{j+1})\|^2_{H^2}\nonumber\\
&+cM_0(\|\nabla(\varrho^{j+1},u^{j+1},w^{j+1},b^{j+1})\|^2_{H^2}+\|\nabla^2u^{j+1}\|^2_{H^2}).
\end{align}
Thus, by linear combination of (\ref{3.67}) and (\ref{3.68}), we deduce that there exists an instant energy functional
$\|\tilde{\mathscr{A}}^{j+1}\|^2_{H^3}$ is equivalent to $\|\mathscr{A}^{j+1}\|^2_{H^3}$ such that
\begin{align}\label{3.69}
\frac{d}{dt}&\|\tilde{\mathscr{A}}^{j+1}(t)\|^2_{H^3}+\lambda\|\nabla(\nabla\Phi^{j+1},u^{j+1},w^{j+1},b^{j+1})\|^2_{H^3}
+\lambda\|(\varrho^{j+1},w^{j+1})\|^2_{H^3} \nonumber\\
 \leq& c M_0\|(\nabla^2\Phi^{j+1},\varrho^{j+1},\nabla u^{j+1},\nabla w^{j+1},\nabla b^{j+1})\|^2_{H^3}
+cM_0^2(\|\nabla( u^{j+1}, b^{j+1})\|^2_{H^3}+\|\nabla\varrho^{j+1}\|^2_{H^2})\nonumber\\
&+ cM_0(\|\nabla\mathscr{A}^{j+1}\|^2_{H^2}+\|\nabla^2u^{j+1}\|^2_{H^2}+\|w^{j+1}\|^2_{H^3})+cM_0\|\tilde{\mathscr{A}}^{j+1}(t)\|_{H^3}^2,
\end{align}
for some $\lambda\in (0,1)$. Further, by the Gronwall's inequality and the property of $\tilde{\mathscr{A}}^{j+1}(t)$, it holds that
\begin{align*}
&\|\mathscr{A}^{j+1}(t)\|^2_{H^3}+\lambda \int_0^t\|\nabla(\nabla\Phi^{j+1},u^{j+1},w^{j+1}, b^{j+1})(s)\|^2_{H^3}
+\|(\varrho^{j+1},w^{j+1})(s)\|^2_{H^3}ds\nonumber  \\
&\leq e^{cM_0t}\left[\varepsilon_1^2+cM_0\int_0^t \|(\nabla^2\Phi^{j+1},\varrho^{j+1},\nabla u^{j+1},\nabla w^{j+1},\nabla b^{j+1})(s)\|^2_{H^3}ds
+cM_0^2\int_0^t\|\nabla(u^{j+1},b^{j+1})(s)\|^2_{H^3}ds\right.\nonumber  \\
&\quad\left. +cM_0\int_0^t(\|\nabla\mathscr{A}^{j+1}(s)\|^2_{H^2}+\|\nabla^2u^{j+1}(s)\|^2_{H^2}+\|w^{j+1}(s)\|^2_{H^3})ds\right],
\end{align*}
for any $0\leq t\leq T_1$. From above,   we can take suitable small $\varepsilon_1>0$, $T_1>0$  and $M_1>0$ such that
\begin{equation}\label{3.70}
\|\mathscr{A}^{j+1}(t)\|^2_{H^3}+\lambda' \int_0^t\|\nabla(\nabla\Phi^{j+1}, u^{j+1}, w^{j+1},b^{j+1})(s)\|^2_{H^3}
+\|(\varrho^{j+1},w^{j+1})(s)\|^2_{H^3}ds\leq M_0^2,
\end{equation}
for any $0\leq t\leq T_1$ and $\lambda'\in(0,\lambda]$ be a constant. This implies that (\ref{3.065}) holds for $j+1$ if so for $j$.
Hence (\ref{3.065}) is proved for all $j\geq 0$.

Next, from (\ref{3.69}) and the equivalence of $\tilde{\mathscr{A}}^{j+1}(t)$ and $\mathscr{A}^{j+1}(t)$, we can see that
\begin{align*}
&\left|\|\mathscr{A}^{j+1}(t)\|^2_{H^3}-\|\mathscr{A}^{j+1}(s)\|^2_{H^3}\right|
=\left|\int_s^t\frac{d}{d\tau}\|\mathscr{A}^{j+1}(\tau)\|^2_{H^3}d\tau\right|  \\
& \leq c\int_s^t\|\mathscr{A}^{j}(\tau)\|_{H^3}\left(\|\nabla(\nabla\Phi^{j+1}, u^{j+1}, w^{j+1},b^{j+1})(\tau)\|^2_{H^3}
+\|(\nabla\Phi^{j+1},\varrho^{j+1},w^{j+1})(\tau)\|^2_{H^3}\right)d\tau\\
&\leq cM_0\int_s^t\|\nabla(\nabla\Phi^{j+1}, u^{j+1}, w^{j+1},b^{j+1})(\tau)\|^2_{H^3}
 +\|(\nabla\Phi^{j+1},\varrho^{j+1},w^{j+1})(\tau)\|^2_{H^3}d\tau,
\end{align*}
for any $0\leq s\leq t\leq T_1.$ Here, the time integral in the last inequality is finite due to (\ref{3.70}), and hence $\|\mathscr{A}^{j+1}(t)\|^2_{H^3}$ is
continuous in $t$ for each $j\geq 0$.

In order to consider the convergence of the sequence $(\mathscr{A}^j)_{j\geq 0}$, by taking the difference of (\ref{3.143}) for each $j$ and $j-1$,
we arrive at
\begin{align}\label{3.71}
\begin{cases}
\partial_t(\varrho^{j+1}-\varrho^{j})+\textrm{div}~(u^{j+1}-u^{j})=M_1^{j+1}-M_1^{j},& \\
\partial_t(u^{j+1}-u^{j})+\nabla(\varrho^{j+1}-\varrho^{j})-\Delta(u^{j+1}-u^{j})-\nabla \textrm{div}~(u^{j+1}-u^{j})&\\
- \nabla\times (w^{j+1}-w^{j})-\nabla(\Phi^{j+1}-\Phi^{j})=M_2^{j+1}-M_2^{j},&\\
\partial_t(w^{j+1}-w^{j})+2 (w^{j+1}-w^{j})-\Delta (w^{j+1}-w^{j})-2\nabla \textrm{div}~(w^{j+1}-w^{j})&\\
-\nabla \times(u^{j+1}-u^{j})=M_3^{j+1}-M_3^{j},&\\
\partial_t(b^{j+1}-b^{j})-\Delta (b^{j+1}-b^{j})= M_4^{j+1}-M_4^{j},&\\
\Delta(\Phi^{j+1}-\Phi^{j})=\varrho^{j+1}-\varrho^{j},&\\
\textrm{div}~(b^{j+1}-b^j)=0,\quad t>0, x\in \mathbb{R}^3,
\end{cases}
\end{align}
where the nonlinear terms $M_i^{j+1}-M_i^j$ $(i=1,2,3,4)$ are defined as
\begin{align*}
M_1^{j+1}-M_1^j=&-[\varrho^{j}\textrm{div}~( u^{j+1}-u^j)+(\varrho^j-\varrho^{j-1})\textrm{div}~u^j+u^j\cdot\nabla(\varrho^{j+1}-\varrho^j)
+(u^j-u^{j-1})\cdot\nabla\varrho^j],\\
M_2^{j+1}-M_2^j=&-[u^j\cdot \nabla( u^{j+1}-u^j)+( u^{j}-u^{j-1})\cdot\nabla u^j]
-f(\varrho^j)[\Delta( u^{j+1}-u^j)+\nabla \textrm{div}~( u^{j+1}-u^j)\\
&+ \nabla \times ( w^{j+1}-w^j)]
-[f(\varrho^j)-f(\varrho^{j-1})](\Delta u^j+\nabla \textrm{div}~ u^j+\nabla\times w^j)
-h(\varrho^j)\nabla(\varrho^{j+1}-\varrho^{j})\\
&-(h(\varrho^j)-h(\varrho^{j-1}))\nabla \varrho^j
+g(\varrho^j)[\nabla\times (b^{j+1}-b^j)\times b^j+\nabla\times b^j\times(b^j-b^{j-1})]\\
&+(g(\varrho^j)-g(\varrho^{j-1}))\nabla\times b^j\times b^{j-1},\\
M_3^{j+1}-M_3^j=&-u^j\cdot \nabla (w^{j+1}-w^j)-(u^j-u^{j-1})\cdot\nabla w^j\\
&-f(\varrho^j)[\Delta (w^{j+1}-w^j)+\nabla \textrm{div}~(w^{j+1}-w^j)-2(w^{j+1}-w^j)+\nabla \times ( u^{j+1}-u^j)]\\
&-(f(\varrho^j)-f(\varrho^{j-1}))[\Delta w^{j}+\nabla \textrm{div}~w^{j}-2 w^{j}+\nabla \times u^{j}],\\
M_4^{j+1}-M_4^j=&b^j\cdot \nabla (u^{j+1}-u^j)+(b^j-b^{j-1})\cdot\nabla u^j-u^j\cdot \nabla (b^{j+1}-b^j)\\
&-(u^j-u^{j-1})\cdot\nabla b^j-b^j\textrm{div}(u^{j+1}-u^j)-(b^{j}-b^{j-1})\textrm{div}~ u^j,
\end{align*}
By using the same energy estimates as before, we infer that
\begin{align*}
\frac{d}{dt}&\|\mathscr{A}^{j+1}-\mathscr{A}^j\|^2_{H^3}
+\lambda''\|\nabla(\nabla(\Phi^{j+1}-\Phi^j),u^{j+1}-u^j,w^{j+1}-w^j,b^{j+1}-b^j)\|^2_{H^3}\\
&+\lambda''\|(\varrho^{j+1}-\varrho^j,w^{j+1}-w^j)\|^2_{H^3}\\
\leq & c\|\mathscr{A}^j-\mathscr{A}^{j-1}\|^2_{H^3}\|(\nabla^2\Phi^j,\varrho^j,\nabla u^j,\nabla w^j,\nabla b^j)\|_{H^3}
+c\|(\varrho^j,u^j)\|_{H^3}\|\nabla(\Phi^{j+1}-\Phi^{j})\|^2_{H^3}\\
&+c\|\mathscr{A}^j\|_{H^3}\|(\nabla^2(\Phi^{j+1}-\Phi^j),\varrho^{j+1}-\varrho^j,\nabla(u^{j+1}-u^j),\nabla(w^{j+1}-w^j),\nabla(b^{j+1}-b^j))\|^2_{H^3}
\end{align*}
with $\lambda''\in (0,1)$ be a constant. Taking into account (\ref{3.70}), the further time integration gives
\begin{align*}
&\|\mathscr{A}^{j+1}(t)-\mathscr{A}^j(t)\|^2_{H^3}
+\lambda''\int_0^t\|\nabla(\nabla(\Phi^{j+1}-\Phi^j),u^{j+1}-u^j,w^{j+1}-w^j,b^{j+1}-b^j)(s)\|^2_{H^3}ds\\
&+\lambda''\int_0^t\|(\varrho^{j+1}-\varrho^j,w^{j+1}-w^j)(s)\|^2_{H^3}ds
\leq e^{ cM_0t}\left[cM_0\sup_{0\leq s\leq T_1}\|\mathscr{A}^j(s)-\mathscr{A}^{j-1}(s)\|^2_{H^3}\right.\\
&\left.+cM_0\int_0^t
\|(\nabla^2(\Phi^{j+1}-\Phi^j),\varrho^{j+1}-\varrho^j,\nabla(u^{j+1}-u^j),\nabla(w^{j+1}-w^j),\nabla(b^{j+1}-b^j))(s)\|^2_{H^3}ds\right],
\end{align*}
for any $0\leq t\leq T_1.$ By the smallness of $M_0$ and $T_1$, then there exists a constant $\beta\in (0,1)$ such that
\begin{equation}\label{3.72}
\sup_{0\leq t\leq T_1}\|\mathscr{A}^{j+1}(t)-\mathscr{A}^{j}(t)\|^2_{H^3}
\leq \beta\sup_{0\leq t\leq T_1}\|\mathscr{A}^j(t)-\mathscr{A}^{j-1}(t)\|^2_{H^3}
\end{equation}
for any $j\geq 1$. From the previous inequality, we conclude that $(\mathscr{A}^j)_{j\geq 0}$ is a Cauchy sequence in the Banach space
$C([0,T_1];H^3)$. Thus, the limit function
\begin{equation*}
\mathscr{A}=\mathscr{A}^0+\lim_{n\longrightarrow\infty}\sum_{j=0}^n(\mathscr{A}^{j+1}-\mathscr{A}^{j})
\end{equation*}
indeed exists in $C([0,T_1];H^3)$ and satisfies
\begin{equation*}
\sup_{0\leq t\leq T_1}\|\mathscr{A}(t)\|_{H^3}\leq \sup_{0\leq t\leq T_1}\|\mathscr{A}^{j}(t)\|_{H^3}\leq M_0,
\end{equation*}
that is (\ref{3.66}).

Finally, suppose that $\mathscr{A}(t)$ and $\tilde{\mathscr{A}}(t)$ are two solutions in $C([0,T_1];H^3)$ satisfying (\ref{3.66}). By applying the
same process as in (\ref{3.72}), we deduce that
\begin{equation*}
\sup_{0\leq t\leq T_1}\|\mathscr{A}(t)-\tilde{\mathscr{A}}(t)\|^2_{H^3}
\leq \beta_1\sup_{0\leq t\leq T_1}\|\mathscr{A}(t)-\tilde{\mathscr{A}}(t)\|^2_{H^3},
\end{equation*}
for $\beta_1\in (0,1)$ be a constant, which implies $\mathscr{A}(t)=\tilde{\mathscr{A}}(t)$ holds. This proves the uniqueness and thus completes the
proof of Theorem \ref{th3.1}.
\end{proof}
\begin{Remark}\label{re3.1}
With a few  modifications to Lemma \ref{le3.3}-\ref{le3.4}, we can deduce the local existence of $\mathscr{A}(t)$ in $H^N$
norm $(N\geq 4)$ without the assumption that $\|\mathscr{A}_0\|_{H^N}$ small enough. In fact, by  re-estimating Lemma \ref{le3.3}-\ref{le3.4} carefully,
we can see that
\begin{align*}
\frac{1}{2} &\frac{d}{dt}\sum_{l=4}^N\|\nabla^l\mathscr{A}^{j+1}(t)\|_{L^2}^2
+\sum_{l=4}^N\|\nabla^{l+1}(\sqrt{\frac23} u^{j+1},\sqrt{3} w^{j+1}, b^{j+1} )\|_{L^2}^2
+\frac12\sum_{l=4}^N\|\nabla^lw^{j+1}\|_{L^2}^2 \nonumber\\
\leq& c\|(\varrho^j,u^j,b^j)\|^{\beta_1}_{H^3}\|(\varrho^j,u^j,b^j)\|^{1-\beta_1}_{H^N}
\sum_{l=4}^N\|(\nabla^{l+2}\Phi^{j+1},\nabla^{l}\varrho^{j+1},\nabla^{l+1} u^{j+1},\nabla^{l+1} w^{j+1},\nabla^{l+1} b^{j+1})\|_{L^2}^2\nonumber\\
&+c\|\varrho^j\|^{\beta_2}_{H^3}\|\varrho^j\|^{1-\beta_2}_{H^N}\|b^j\|^{\beta_3}_{H^3}\|b^j\|^{1-\beta_3}_{H^N}
\sum_{l=4}^N\|\nabla^{l+1}(u^{j+1},b^{j+1})\|_{L^2}^2+c\|\varrho^j\|^{\beta_4}_{H^3}\|\varrho^j\|^{1-\beta_4}_{H^N}
\sum_{l=4}^N\|\nabla^{l}w^{j+1}\|_{L^2}^2,
\end{align*}
and
\begin{align*}
\frac{d}{dt}&\sum_{l=0}^{N-1}\int\nabla^lu^{j+1}\cdot\nabla^{l+1}\varrho^{j+1}dx
+\frac{1}{4}\|(\varrho^{j+1},\nabla\varrho^{j+1},\nabla^2\Phi^{j+1},\nabla^3\Phi^{j+1})\|^2_{H^{N-1}} \nonumber\\
\leq&\|\nabla u^{j+1}\|^2_{H^{N-1}}+2\|\nabla w^{j+1}\|^2_{H^{N-1}}+4\|\nabla^2 u^{j+1}\|^2_{H^{N-1}}
+c\|\varrho^j\|^{\beta_5}_{H^3}\|\varrho^j\|^{1-\beta_5}_{H^N}\|b^j\|^{\beta_6}_{H^3}\|b^j\|^{1-\beta_6}_{H^N}
\|\nabla(\varrho^{j+1},b^{j+1})\|^2_{H^{N-1}}\nonumber\\
&+c\|(\varrho^j,u^j,b^j)\|^{\beta_7}_{H^3}\|(\varrho^j,u^j,b^j)\|^{1-\beta_7}_{H^N}
(\|\nabla(\varrho^{j+1},u^{j+1},w^{j+1},b^{j+1})\|^2_{H^{N-1}}+\|\nabla^2u^{j+1}\|^2_{H^{N-1}}),
\end{align*}
where $\beta_i\in (0,1)$, $i=1,\cdots,7$. For any $M>0$ be fixed. Suppose that $\|\mathscr{A}_0\|_{H^N}\leq M$, then applying the smallness of $M_0$
in Theorem \ref{th3.1}, by the same process as the proof of Theorem \ref{th3.1}, one can verify that $\|\mathscr{A}^j(t)\|_{H^N}\leq M$, for any
$t\in [0,T_1]$ and $T_1$ was determined in Theorem \ref{th3.1}. Furthermore, we can also conclude that the limit function of $\mathscr{A}^j(t)$ is indeed
the solution over $[0,T_1]$ to (\ref{1.3})-(\ref{1.4}).
\end{Remark}
\section{Energy evolution of negative Sobolev or Besov norms}\label{se4}
In this section, we shall show the evolution of the negative Sobolev and Besov norms of the solution $(\nabla\Phi,\varrho,u,w,b)$ to
(\ref{1.3})-(\ref{1.4}). In order to estimate the nonlinear terms, we shall restrict ourselves to that $s\in (0,\frac32]$. Firstly, for
the homogeneous Sobolev space, we have the following result.
\begin{Lemma}\label{le4.1}
Let  all assumptions in Lemma \ref{le3.1} are in force. Then, for $s\in (0,\frac{1}{2}]$, we have
\begin{align}\label{4.1}
&\frac{1}{2}\frac{d}{dt}\|(\nabla\Phi,\varrho,u,w,b)\|_{\dot{H}^{-s}}^2+\frac23\|\nabla u\|_{\dot{H}^{-s}}^2+3\|\nabla w\|_{\dot{H}^{-s}}^2
+\frac12\|w\|_{\dot{H}^{-s}}^2+\|\nabla b\|_{\dot{H}^{-s}}^2\nonumber\\
&\lesssim(\|\nabla(\varrho,u,w,b)\|_{H^1}^2+\|(\varrho,w)\|_{L^2}^2)\|(\nabla\Phi,\varrho,u,w,b)\|_{\dot{H}^{-s}},
\end{align}
and for $s\in (\frac{1}{2},\frac{3}{2})$, we have
\begin{align}\label{4.2}
&\frac{1}{2}\frac{d}{dt}\|(\nabla\Phi,\varrho,u,w,b)\|_{\dot{H}^{-s}}^2+\frac23\|\nabla u\|_{\dot{H}^{-s}}^2+3\|\nabla w\|_{\dot{H}^{-s}}^2
+\frac12\|w\|_{\dot{H}^{-s}}^2+\|\nabla b\|_{\dot{H}^{-s}}^2\nonumber\\
&\lesssim\|(\varrho,u,w,b)\|^{s-\frac12}_{L^2}\|\nabla(\varrho,u,w,b)\|^{\frac32-s}_{L^2}
\|(\nabla\varrho,\nabla u,\nabla w,\nabla b,\varrho,w,\nabla^2u,\nabla^2w)\|_{L^2}\|(\nabla\Phi,\varrho,u,w,b)\|_{\dot{H}^{-s}}.
\end{align}
\end{Lemma}
\begin{proof}
By the $\Lambda^{-s}$ $(s>0)$ energy estimate of \eqref{1.3}$_1$-\eqref{1.3}$_4$, we obtain
\begin{align}\label{4.3}
\frac{1}{2}&\frac{d}{dt}\|(\varrho,u,w,b)\|_{\dot{H}^{-s}}^2+2\|\nabla u\|_{\dot{H}^{-s}}^2+3\|\nabla w\|_{\dot{H}^{-s}}^2
+2\|w\|_{\dot{H}^{-s}}^2+\|\nabla b\|_{\dot{H}^{-s}}^2\nonumber\\
=&\int \Lambda^{-s}[-\textrm{div}~u-\textrm{div}~(\varrho u)]\cdot \Lambda^{-s}\varrho dx
+\int \Lambda^{-s}\Big\{\nabla\times w-\nabla \varrho-u\cdot \nabla u-f(\varrho)\big[\Delta u+\nabla\textrm{div}~u+\nabla\times w\big]\nonumber\\
&-h(\varrho)\nabla \varrho+g(\varrho)\big[b\cdot \nabla b-\frac{1}{2}\nabla(|b|^2)\big]+\nabla\Phi\Big\}\cdot \Lambda^{-s}udx\nonumber\\
&+\int\Lambda^{-s}\Big\{\nabla\times u-u\cdot \nabla w
-f(\varrho)\big[\Delta w+2\nabla\textrm{div}~w+\nabla\times u-2 w\big]\Big\}\cdot \Lambda^{-s}w dx\nonumber\\
&+\int\Lambda^{-s}\big[(b\cdot \nabla)u-(u\cdot \nabla)b-b(\textrm{div}~u)\big]\cdot \Lambda^{-s}b dx\nonumber\\
\leq&\|\textrm{div}~(\varrho u)\|_{\dot{H}^{-s}}\|\varrho\|_{\dot{H}^{-s}}+\|-u\cdot\nabla u
-f(\varrho)[\Delta u+\nabla\textrm{div}~u+\nabla\times w]-h(\varrho)\nabla\varrho\|_{\dot{H}^{-s}}\|u\|_{\dot{H}^{-s}}\nonumber\\
&+\|g(\varrho)[b\cdot\nabla b-\frac{1}{2}\nabla(|b|^2)]\|_{\dot{H}^{-s}}\|u\|_{\dot{H}^{-s}}+\int\Lambda^{-s}\nabla\Phi\cdot\Lambda^{-s}udx\nonumber\\
&+\|-u\cdot \nabla w-f(\varrho)[\Delta w+2\nabla\textrm{div}~w+\nabla\times u-2w]\|_{\dot{H}^{-s}}\cdot\| w\|_{\dot{H}^{-s}}\nonumber\\
&+\|(b\cdot \nabla)u-(u\cdot \nabla)b-b(\textrm{div}~u)\|_{\dot{H}^{-s}}\| b\|_{\dot{H}^{-s}}
+\int\Lambda^{-s}\nabla\times w\cdot\Lambda^{-s}udx+\int\Lambda^{-s}\nabla\times u\cdot\Lambda^{-s}wdx,
\end{align}
where in the last inequality, we have taken into account that
\begin{equation*}
\int\Lambda^{-s}(-\textrm{div}~u)\cdot\Lambda^{-s}\varrho dx+\int\Lambda^{-s}(-\nabla\varrho)\cdot\Lambda^{-s}udx=0.
\end{equation*}
Now, we concentrate our attention on estimating  the terms on the right hand side of \eqref{4.3}, and we distinguish the argument by the value of $s$.
First, if $s\in (0,\frac{1}{2}]$, then $\frac{1}{2}+\frac{s}{3}<1$ and $\frac{3}{s} \geq 6$.

In virtue of \eqref{2.6}, Sobolev inequality,  the H\"{o}lder and Young's inequality, we obtain
\begin{align}\label{4.4}
\|\textrm{div}(\varrho u)\|_{\dot{H}^{-s}}
&\leq \|\varrho\textrm{div}~u\|_{\dot{H}^{-s}}+\|\nabla\varrho\cdot u\|_{\dot{H}^{-s}}\nonumber\\
&\lesssim \|\varrho\textrm{div}~u\|_{L^{\frac{1}{\frac12+\frac s3}}}+\|\nabla\varrho\cdot u\|_{L^{\frac{1}{\frac12+\frac s3}}}\nonumber\\
&\lesssim \|\varrho\|_{L^{\frac{3}{s}}}\|\nabla u\|_{L^2}+\| u\|_{L^{\frac{3}{s}}}\|\nabla\varrho\|_{L^2}\nonumber\\
&\lesssim\|\nabla\varrho\|_{L^2}^{\frac{1}{2}+s}\|\nabla^2\varrho\|_{L^2}^{\frac{1}{2}-s}\|\nabla u\|_{L^2}
+\|\nabla u\|_{L^2}^{\frac{1}{2}+s}\|\nabla^2 u\|_{L^2}^{\frac{1}{2}-s}\|\nabla\varrho\|_{L^2}\nonumber\\
&\lesssim\|\nabla\varrho\|_{H^1}^2+\|\nabla u\|^2_{H^1},
\end{align}
Similarly,  we further obtain
\begin{align}
\| u\cdot\nabla u\|_{\dot{H}^{-s}}&\lesssim \|\nabla u\|_{H^1}^2,\label{4.5}\\
\| u\cdot\nabla w\|_{\dot{H}^{-s}}&\lesssim \|\nabla w \|_{L^2}^2+\|\nabla u\|_{H^1}^2,\label{4.6}\\
\|b\cdot\nabla u\|_{\dot{H}^{-s}}&\lesssim \|\nabla u\|_{L^2}^2+\|\nabla b\|_{H^1}^2,\label{4.7}\\
\|u\cdot\nabla b\|_{\dot{H}^{-s}}&\lesssim\|\nabla b\|_{L^2}^2+\|\nabla u\|_{H^1}^2,\label{4.8}\\
\|b\textrm{div}~u\|_{\dot{H}^{-s}}&\lesssim \|\nabla u\|_{L^2}^2+\|\nabla b\|_{H^1}^2.\label{4.9}
\end{align}
In addition, using Lemma \ref{le2.3}-\ref{le2.4}, there holds
\begin{equation}\label{4.10}
\|f(\varrho)\Delta u\|_{\dot{H}^{-s}}+\| f(\varrho)\nabla\textrm{div}~u\|_{\dot{H}^{-s}}
\lesssim \|f(\varrho)\|_{L^{\frac{3}{s}}}\|\Delta u\|_{L^2}
\lesssim \|\nabla \varrho\|_{H^1}+\|\nabla u\|_{H^1}^2.
\end{equation}
By the same way, we infer that
\begin{align}
\|f(\varrho)\nabla\times w\|_{\dot{H}^{-s}}&\lesssim\|\nabla\varrho\|_{H^1}^2+\|\nabla w\|_{L^2}^2,\label{4.11}\\
\|f(\varrho)\nabla\times u\|_{\dot{H}^{-s}}&\lesssim\|\nabla u\|_{L^2}^2+\|\nabla \varrho\|_{H^1}^2,\label{4.12}\\
\|f(\varrho)w\|_{\dot{H}^{-s}}&\lesssim\|w\|_{L^2}^2+\|\nabla\varrho\|_{H^1}^2,\label{4.13}\\
\|f(\varrho)\Delta w\|_{\dot{H}^{-s}}+\|f(\varrho)\nabla\textrm{div}~w\|_{\dot{H}^{-s}}
&\lesssim\|\nabla\varrho\|_{H^1}^2+\|\nabla w\|_{H^1}^2,\label{4.14}\\
\|h(\varrho)\nabla\varrho\|_{\dot{H}^{-s}}&\lesssim \|\nabla \varrho\|_{H^1}^2,\label{4.15}
\end{align}
Next, note that $g(\varrho)\in (0,1)$, then we can see that
\begin{align}\label{4.16}
\|g(\varrho)(b\cdot\nabla b)\|_{\dot{H}^{-s}}&\lesssim\|g(\varrho)b\|_{L^{\frac3s}}\|\nabla b\|_{L^2}\nonumber\\
&\lesssim\|b\|_{L^{\frac3s}}\|\nabla b\|_{L^2}\lesssim\|\nabla b\|^2_{H^1}.
\end{align}
Likewise, there holds
\begin{equation}\label{4.17}
\|g(\varrho)\cdot\nabla(|b|^2)\|_{\dot{H}^{-s}}\lesssim \|\nabla b\|_{H^1}^2.
\end{equation}
Moreover, taking into account (\ref{3.040}), by the Young's inequality, we have
\begin{align}\label{4.18}
&\int\Lambda^{-s}\nabla\times w\cdot\Lambda^{-s}udx+\int\Lambda^{-s}\nabla\times u\cdot\Lambda^{-s}wdx\nonumber \\
&=2\int\Lambda^{-s}\nabla\times u\cdot\Lambda^{-s}wdx
\leq\frac43\|\Lambda^{-s}\nabla u\|_{L^2}^2+\frac32\|\Lambda^{-s}w\|_{L^2}^2.
\end{align}
Finally, for the Poisson term, by integration by parts, it holds that
\begin{align}\label{4.19}
-\int\Lambda^{-s}\nabla\Phi\cdot\Lambda^{-s}udx&=\int\Lambda^{-s}\Phi\Lambda^{-s}\textrm{div}~udx \nonumber \\
& =\int\Lambda^{-s}\Phi\Lambda^{-s}(-\partial_t\varrho-\textrm{div}~(\varrho u))dx \nonumber \\
& =\int-\Lambda^{-s}\Phi\Lambda^{-s}\partial_t\Delta\Phi+\Lambda^{-s}\nabla\Phi\cdot\Lambda^{-s}(\varrho u)dx \nonumber \\
&=\frac12\frac d{dt}\int|\Lambda^{-s}\nabla\Phi|^2dx+\int\Lambda^{-s}\nabla\Phi\cdot\Lambda^{-s}(\varrho u)dx.
\end{align}
If $s\in(0,\frac 12)$, using Lemma \ref{le2.1} and Lemma \ref{le2.4}, we infer that
\begin{equation}\label{4.20}
\|\Lambda^{-s}(\varrho u)\|_{L^2}\lesssim \|\varrho\|_{L^2}\|u\|_{L^{\frac 3s}}\lesssim \|\varrho\|^2_{L^2}+\|\nabla u\|^2_{H^1}.
\end{equation}
Thus, from (\ref{4.3})-(\ref{4.20}), it follows (\ref{4.1}) for $s\in(0,\frac12)$.

On the other hand, if $s\in (\frac{1}{2},\frac{3}{2})$, we have $\frac{1}{2}+\frac{s}{3}<1$, $2<\frac{3}{s}<6.$ Then, we obtain
\begin{align}\label{4.21}
\|\textrm{div}~(\varrho u)\|_{\dot{H}^{-s}}
&\leq\|\nabla\varrho\cdot u\|_{\dot{H}^{-s}}+\|\varrho\textrm{div}u\|_{\dot{H}^{-s}}\nonumber\\
&\leq \| u\|_{L^{\frac{3}{s}}}\|\nabla \varrho\|_{L^2}+\|\varrho\|_{L^{\frac{3}{s}}}\|\nabla u\|_{L^2}\nonumber\\
&\lesssim\|u\|_{L^2}^{s-\frac{1}{2}}\|\nabla u\|_{L^2}^{\frac{3}{2}-s}\|\nabla\varrho\|_{L^2}
+\|\varrho\|_{L^2}^{s-\frac{1}{2}}\|\nabla \varrho\|_{L^2}^{\frac{3}{2}-s}\|\nabla u\|_{L^2},
\end{align}
Likewise, we can also obtain
\begin{align}
\|u\cdot \nabla u\|_{\dot{H}^{-s}}&\lesssim\|u\|_{L^2}^{s-\frac{1}{2}}\|\nabla u \|_{L^2}^{\frac{3}{2}-s}\|\nabla u\|_{L^2},\label{4.22}\\
\|u\cdot \nabla w\|_{\dot{H}^{-s}}&\lesssim\|u\|_{L^2}^{s-\frac{1}{2}}\|\nabla u\|_{L^2}^{\frac{3}{2}-s}\|\nabla w\|_{L^2},\label{4.23}\\
\|(b\cdot\nabla)u\|_{\dot{H}^{-s}}&\lesssim\|b\|_{L^2}^{s-\frac{1}{2}}\|\nabla b\|_{L^2}^{\frac{3}{2}-s}\|\nabla u\|_{L^2},\label{4.24}\\
\|(u\cdot\nabla)b\|_{\dot{H}^{-s}}&\lesssim\|u\|_{L^2}^{s-\frac{1}{2}}\|\nabla u\|_{L^2}^{\frac{3}{2}-s}\|\nabla b\|_{L^2},\label{4.25}\\
\|b\textrm{div}~u\|_{\dot{H}^{-s}}&\lesssim\|b\|_{L^2}^{s-\frac{1}{2}}\|\nabla b\|_{L^2}^{\frac{3}{2}-s}\|\nabla u\|_{L^2}.\label{4.26}
\end{align}
Taking into account Lemma \ref{le2.3}, we further obtain
\begin{align}
\|f(\varrho)\Delta u\|_{\dot{H}^{-s}}+\|f(\varrho)\nabla\textrm{div}~u\|_{\dot{H}^{-s}}
&\lesssim \| \varrho\|_{L^2}^{s-\frac{1}{2}}\|\nabla \varrho\|_{L^2}^{\frac{3}{2}-s}\|\nabla^2 u\|_{L^2},\label{4.27}\\
\|f(\varrho)\Delta w\|_{\dot{H}^{-s}}+\|f(\varrho)\nabla\textrm{div}~w\|_{\dot{H}^{-s}}
&\lesssim \| \varrho\|_{L^2}^{s-\frac{1}{2}}\|\nabla\varrho\|_{L^2}^{\frac{3}{2}-s}\|\nabla^2 w\|_{L^2},\label{4.28}\\
\|f(\varrho)\nabla\times w\|_{\dot{H}^{-s}}+\|f(\varrho)\nabla \times u\|_{\dot{H}^{-s}}
&\lesssim\|\varrho\|_{L^2}^{s-\frac{1}{2}}\|\nabla \varrho\|_{L^2}^{\frac{3}{2}-s}(\|\nabla w\|_{L^2}+\|\nabla u\|_{L^2}),\label{4.29}\\
\|f(\varrho)w \|_{\dot{H}^{-s}}+\|h(\varrho)\nabla \varrho\|_{\dot{H}^{-s}}
&\lesssim\| \varrho\|_{L^2}^{s-\frac{1}{2}}\|\nabla \varrho\|_{L^2}^{\frac{3}{2}-s}(\| w\|_{L^2}+\|\nabla\varrho\|_{L^2}),\label{4.31}
\end{align}
In addition, similar to (\ref{4.16})-(\ref{4.17}), there holds
\begin{equation}\label{4.33}
\|g(\varrho)(b\cdot\nabla b)\|_{\dot{H}^{-s}}+ \|g(\varrho)\cdot\nabla(|b|^2)\|_{\dot{H}^{-s}}
\lesssim \|b\|_{L^2}^{s-\frac{1}{2}}\|\nabla b\|_{L^2}^{\frac{3}{2}-s}\|\nabla b\|_{L^2}.
\end{equation}
While for the right-most term in (\ref{4.19}), we can see that
\begin{equation}\label{4.34}
\|\Lambda^{-s}(\varrho u)\|_{L^2}\lesssim \|\varrho\|_{L^2}\|u\|_{L^{\frac 3s}}
\lesssim \|\varrho\|^2_{L^2}\| u\|^{s-\frac12}_{L^2}\|\nabla u\|^{\frac32-s}_{L^2}.
\end{equation}
Combining (\ref{4.3}), (\ref{4.18})-(\ref{4.19}) and (\ref{4.21})-(\ref{4.34}), we have (\ref{4.2}) for $s\in (\frac12,\frac32)$. This completes
the proof of Lemma \ref{le4.1}.
\end{proof}
When replace the homogeneous Sobolev space by the homogeneous Besov space. Now, we  proceed to derive the evolution of the negative Besov norms of the
solution $(\nabla\Phi,\varrho,u,w,b)$ to (\ref{1.3})-\eqref{1.4}.  Precisely, we have
\begin{Lemma}\label{le4.2}
Let all assumptions in Lemma \ref{le3.1} hold. Then, for $s\in (0,\frac{1}{2}]$, we have
\begin{align}\label{4.35}
&\frac{1}{2}\frac{d}{dt}\|(\nabla\Phi,\varrho,u,w,b)\|_{\dot{B}_{2,\infty}^{-s}}^2
+\frac23\|\nabla u\|_{\dot{B}_{2,\infty}^{-s}}^2+3\|\nabla w\|_{\dot{B}_{2,\infty}^{-s}}^2
+\frac12\|w\|_{\dot{B}_{2,\infty}^{-s}}^2+\|\nabla b\|_{\dot{B}_{2,\infty}^{-s}}^2\nonumber\\
&\lesssim(\|\nabla(\varrho,u,w,b)\|_{H^1}^2+\|(\varrho,w)\|_{L^2}^2)\|(\nabla\Phi,\varrho,u,w,b)\|_{\dot{B}_{2,\infty}^{-s}},
\end{align}
and for $s\in (\frac{1}{2},\frac{3}{2}]$, we have
\begin{align}\label{4.36}
&\frac{1}{2}\frac{d}{dt}\|(\nabla\Phi,\varrho,u,w,b)\|_{\dot{B}_{2,\infty}^{-s}}^2
+\frac23\|\nabla u\|_{\dot{B}_{2,\infty}^{-s}}^2+3\|\nabla w\|_{\dot{B}_{2,\infty}^{-s}}^2
+\frac12\| w\|_{\dot{B}_{2,\infty}^{-s}}^2+\|\nabla b\|_{\dot{B}_{2,\infty}^{-s}}^2\nonumber\\
&\lesssim\|(\varrho,u,w,b)\|^{s-\frac12}_{L^2}\|\nabla(\varrho,u,w,b)\|^{\frac32-s}_{L^2}
\|(\nabla\varrho,\nabla u,\nabla w,\nabla b,\varrho,w,\nabla^2u,\nabla^2w)\|_{L^2}\|(\nabla\Phi,\varrho,u,w,b)\|_{\dot{B}_{2,\infty}^{-s}}.
\end{align}
\end{Lemma}
\begin{proof}
Applying $\dot{\Delta}_j$ energy estimate of \eqref{1.3}$_1$-\eqref{1.3}$_4$ with multiplication of $2^{-2sj}$ and then taking the supremum over
$j\in \mathbb{Z}$, we infer that
\begin{align}\label{4.37}
\frac{1}{2}&\frac{d}{dt}\|(\varrho,u,w,b)\|_{\dot{B}_{2,\infty}^{-s}}^2+2\|\nabla u\|_{\dot{B}_{2,\infty}^{-s}}^2
+3\|\nabla w\|_{\dot{B}_{2,\infty}^{-s}}^2+2\|w\|_{\dot{B}_{2,\infty}^{-s}}^2+\|\nabla b\|_{\dot{B}_{2,\infty}^{-s}}^2\nonumber\\
\lesssim&\sup_{j\in \mathbb{Z}}2^{-2sj}\int \dot{\Delta}_{j}[-\textrm{div}~u-\textrm{div}~(\varrho u)]\cdot \dot{\Delta}_{j}\varrho dx\nonumber\\
&+\sup_{j\in \mathbb{Z}}2^{-2sj}\int \dot{\Delta}_{j}\Big\{\nabla\times w-\nabla \varrho-u\cdot \nabla u
-f(\varrho)\big[\Delta u+\nabla\textrm{div}~u+\nabla\times w\big]\nonumber\\
&-h(\varrho)\nabla \varrho+g(\varrho)\big[b\cdot \nabla b-\frac{1}{2}\nabla(|b|^2)\big]+\nabla\Phi\Big\}\cdot \dot{\Delta}_{j}udx\nonumber\\
&+\sup_{j\in \mathbb{Z}}2^{-2sj}\int\dot{\Delta}_{j}\Big\{\nabla\times u-u\cdot \nabla w
-f(\varrho)\big[\Delta w+2\nabla\textrm{div}~w+\nabla\times u-2 w\big]\Big\}\cdot \dot{\Delta}_{j}w dx\nonumber\\
&+\sup_{j\in \mathbb{Z}}2^{-2sj}\int\dot{\Delta}_{j}\big[(b\cdot \nabla)u-(u\cdot \nabla)b-b(\textrm{div}~u)\big]\cdot \dot{\Delta}_{j}b dx\nonumber
\end{align}
\begin{align}
\leq&\|\textrm{div}~(\varrho u)\|_{\dot{B}_{2,\infty}^{-s}}\|\varrho\|_{\dot{B}_{2,\infty}^{-s}}
+\|-u\cdot\nabla u-f(\varrho)[\Delta u+\nabla\textrm{div}~u+\nabla\times w]-h(\varrho)\nabla\varrho\|_{\dot{B}_{2,\infty}^{-s}}
\|u\|_{\dot{B}_{2,\infty}^{-s}}\nonumber\\
&+\|g(\varrho)[b\cdot\nabla b-\frac{1}{2}\nabla(|b|^2)]\|_{\dot{B}_{2,\infty}^{-s}}\|u\|_{\dot{B}_{2,\infty}^{-s}}
+\sup_{j\in \mathbb{Z}}2^{-2sj}\int\dot{\Delta}_{j}\nabla\Phi\cdot\dot{\Delta}_{j}udx\nonumber\\
&+\|-u\cdot \nabla w-f(\varrho)[\Delta w+2\nabla\textrm{div}~w+\nabla\times u-2w]\|_{\dot{B}_{2,\infty}^{-s}}\|w\|_{\dot{B}_{2,\infty}^{-s}}\nonumber\\
&+\|(b\cdot \nabla)u-(u\cdot \nabla)b-b(\textrm{div}~u)\|_{\dot{B}_{2,\infty}^{-s}}\| b\|_{\dot{B}_{2,\infty}^{-s}}\nonumber\\
&+\sup_{j\in \mathbb{Z}}2^{-2sj}\int\dot{\Delta}_{j}\nabla\times w\cdot\dot{\Delta}_{j}udx
+\sup_{j\in \mathbb{Z}}2^{-2sj}\int\dot{\Delta}_{j}\nabla\times u\cdot\dot{\Delta}_{j}wdx,
\end{align}
According to  the Lemma \ref{le2.5} and \eqref{4.37}, the remaining proof of Lemma \ref{le4.2} is exactly same with the proof of Lemma \ref{le4.1},
except that we allow $s=\frac{3}{2}$ and replace Lemma \ref{le2.4} with Lemma \ref{le2.5}, $\dot{H}^{-s}$ norm by $\dot{B}_{2,\infty}^{-s}$ norm.
\end{proof}
\section{Proof of main theorem}\label{se5}
In this section, based on the assumption that $H^3$ norm of initial data is small,  we shall combine all energy estimates that we have derived in the
previous two sections to prove the global existence of $(\nabla\Phi,\varrho,u,w,b)$ to (\ref{1.3})-(\ref{1.4}).
\begin{proof}[Proof of Theorem \ref{th1.1}]
For simplicity, we divide the proof into several steps.

\textbf{Step 1}. Global small $\mathscr{E}_3$-solution.

We first close the energy estimates at the $H^3$-level by assuming a priori that $\sqrt{\mathscr{E}_3(t)}\leq\delta$ is sufficiently small. Thus, from
Lemma \ref{le3.1}, taking $k=0,1$ in (\ref{3.2}) and summing up, we deduce that for any $t\in [0,T]$
\begin{align}\label{5.1}
&\frac{1}{2}\frac{d}{dt}\sum_{l=0}^{3}\|\nabla^l(\nabla\Phi,\varrho,u,w,b)\|_{L^2}^{2}+\frac12\sum_{l=0}^{3}\|\nabla^lw\|_{L^2}^{2}
+\sum_{l=0}^{3}(\frac23\|\nabla^{l+1}u\|_{L^2}^{2}+3\|\nabla^{l+1}w\|_{L^2}^{2}+\|\nabla^{l+1}b\|_{L^2}^{2})\nonumber\\
&\leq c_3\delta\sum_{l=0}^{3}\left(\|\nabla^l(\varrho,w)\|_{L^2}^{2}+\|\nabla^{l+1}(\nabla\Phi,u,w,b)\|_{L^2}^{2}\right).
\end{align}
In addition, taking $k=0,1$ in (\ref{3.41}) of Lemma \ref{le3.2} and summing up, we obtain
\begin{align}\label{5.2}
&\frac{d}{dt}\sum_{l=0}^{2}\int \nabla^l u\cdot \nabla^{l+1}\varrho dx
+\frac{1}{4}\sum_{l=0}^{2}(\|\nabla^{l}\varrho\|_{L^2}^2+\|\nabla^{l+1}\varrho\|_{L^2}^2+\|\nabla^{l+1}\nabla\Phi\|_{L^2}^2+
\|\nabla^{l+2}\nabla\Phi\|_{L^2}^2)\nonumber\\
\leq&c_2\delta\sum_{l=0}^{2}(\|\nabla^{l+1}\varrho\|_{L^2}^2+\|\nabla^{l+2}(w,b)\|_{L^2}^2)
+\sum_{l=0}^{2}(\|\nabla^{l+1}u\|_{L^2}^2+2\|\nabla^{l+1}w\|_{L^2}^2+4\|\nabla^{l+2}u\|_{L^2}^2).
\end{align}
Taking into account the smallness of $\delta$, by linear combination of (\ref{5.1}) and (\ref{5.2}), we deduce that there exists an instant energy
functional $\tilde{\mathscr{E}}_3$ is equivalent to $\mathscr{E}_3$ such that
\begin{equation}\label{5.3}
\tilde{\mathscr{E}}_3(t)+\int_0^t\mathscr{D}_3(s)ds\leq \tilde{c}_3 \tilde{\mathscr{E}}_3(0), \quad \forall t\in [0,T].
\end{equation}
In what follows, we denote $\tilde{\mathscr{E}}_3(t)$ by $\mathscr{E}_3(t)$ due to the equivalence of $\tilde{\mathscr{E}}_3(t)$ and $\mathscr{E}_3(t)$.
Now, choose a positive constant
\begin{equation*}
\varepsilon_0:=\min\left\{\delta,\varepsilon_1\right\},
\end{equation*}
where $\delta$ and $\varepsilon_1$ are given in Lemma \ref{le3.1} and Theorem \ref{th3.1}, respectively. Further, choose initial data
$(\nabla\Phi(0),\varrho_0,u_0,w_0,b_0)$ and small constant $\delta_0$ such that
\begin{equation*}
\sqrt{\mathscr{E}_3(0)}\leq\sqrt{\delta_0}:=\frac{\varepsilon_0}{2\sqrt{1+\tilde{c}_3}}.
\end{equation*}
Define the lifespan to Cauchy problem (\ref{1.3})-(\ref{1.4}) by
\begin{equation*}
T:=\sup\left\{t:\ \sup_{0\leq s\leq t}\sqrt{\mathscr{E}_3(s)}\leq\varepsilon_0\right\}.
\end{equation*}
Since
\begin{equation*}
\sqrt{\mathscr{E}_3(0)}\leq \frac{\varepsilon_0}{2\sqrt{1+\tilde{c}_3}}\leq\frac{\varepsilon_0}{2}<\varepsilon_0\leq\varepsilon_1,
\end{equation*}
then $T>0$ holds true from the local existence result Theorem \ref{th3.1} and the continuation argument. If $T$ is finite, as a consequence, from the
definition of $T$, it follows that
\begin{equation*}
\sup_{0\leq s\leq T}\sqrt{\mathscr{E}_3(s)}=\varepsilon_0,
\end{equation*}
which is a contradiction to the fact from uniform a priori that
\begin{equation*}
\sqrt{\mathscr{E}_3(s)}\leq\sqrt{\tilde{c}_3}\sqrt{\mathscr{E}_3(0)}\leq \frac{\varepsilon_0\sqrt{\tilde{c}_3}}{2\sqrt{1+\tilde{c}_3}}
\leq\frac{\varepsilon_0}{2}.
\end{equation*}
Therefore, $T=\infty$. This implies that the local solution $(\nabla\Phi,\varrho,u,w,b)$ obtained in Theorem \ref{th3.1} can be extent to infinite time.
Thus, the Cauchy problem (\ref{1.3})-(\ref{1.4}) admits a unique solution $(\nabla\Phi,\varrho,u,w,b)\in C([0,\infty];H^3)$. Finally, (\ref{1.11}) follows
from (\ref{5.3}). This proves the existence of unique global $\mathscr{E}_3$-solution.

\textbf{Step 2.} Global $\mathscr{E}_N$ solution.

From Remark \ref{re3.1} and the global existence of $\mathscr{E}_3$ solution, we can deduce  the global existence of $\mathscr{E}_N$ solution. Thus,
for $N\geq 4$, $t\in [0,\infty]$, applying Lemma \ref{le3.1} and taking $k=0,\cdots,N-2$, we infer that
\begin{align}\label{5.4}
&\frac{1}{2}\frac{d}{dt}\sum_{l=0}^{N}\|\nabla^l(\nabla\Phi,\varrho,u,w,b)\|_{L^2}^{2}+\frac12\sum_{l=0}^{N}\|\nabla^lw\|_{L^2}^{2}
+\sum_{l=0}^{N}(\frac23\|\nabla^{l+1}u\|_{L^2}^{2}+3\|\nabla^{l+1}w\|_{L^2}^{2}+\|\nabla^{l+1}b\|_{L^2}^{2})\nonumber\\
&\leq c_N\varepsilon_0\sum_{l=0}^{N}\left(\|\nabla^l(\varrho,w)\|_{L^2}^{2}+\|\nabla^{l+1}(\nabla\Phi,u,w,b)\|_{L^2}^{2}\right).
\end{align}
Furthermore, by Lemma \ref{le3.2} and  taking $k=0,\cdots,N-2$, we have
\begin{align}\label{5.5}
&\frac{d}{dt}\sum_{l=0}^{N-1}\int \nabla^l u\cdot \nabla^{l+1}\varrho dx
+\frac{1}{4}\sum_{l=0}^{N-1}(\|\nabla^{l}\varrho\|_{L^2}^2+\|\nabla^{l+1}\varrho\|_{L^2}^2+\|\nabla^{l+1}\nabla\Phi\|_{L^2}^2+
\|\nabla^{l+2}\nabla\Phi\|_{L^2}^2)\nonumber\\
\leq&c_{N-1}\varepsilon_0\sum_{l=0}^{N-1}(\|\nabla^{l+1}\varrho\|_{L^2}^2+\|\nabla^{l+2}(w,b)\|_{L^2}^2)
+\sum_{l=0}^{N-1}(\|\nabla^{l+1}u\|_{L^2}^2+2\|\nabla^{l+1}w\|_{L^2}^2+4\|\nabla^{l+2}u\|_{L^2}^2).
\end{align}
By linear combination of (\ref{5.4}) and (\ref{5.5}), we infer that there exists an instant energy functional $\tilde{\mathscr{E}}_N$ is equivalent to
$\mathscr{E}_N$, such that
\begin{equation}\label{5.04}
\frac{d}{dt}\tilde{\mathscr{E}}_N+\tilde{\lambda}\mathscr{D}_N(t)\leq 0,
\end{equation}
for some $\tilde{\lambda}\in (0,1)$. This implies (\ref{1.12}). Thus, we have completed the proof of Theorem \ref{th1.1}.
\end{proof}
According to the conclusion of Theorem \ref{th1.1}, Lemma \ref{le4.1}-\ref{le4.2}, now, we proceed to prove the various time decay rates of the
unique global solution to (\ref{1.3})-(\ref{1.4}).
\begin{proof}[Proof of Theorem \ref{th1.2}]
In what follows, for the convenience of presentations, we define a family of energy functionals and the corresponding dissipation rates as
\begin{equation}\label{5.05}
\mathscr{E}_k^{k+2}:=\sum_{l=k}^{k+2}\|\nabla^l(\nabla\Phi,\varrho,u,w,b)\|^2_{L^2},
\end{equation}
and
\begin{equation}\label{5.06}
\mathscr{D}_k^{k+2}:=\sum_{l=k}^{k+2}\|\nabla^l(\varrho,w)\|^2_{L^2}+\sum_{l=k}^{k+2}\|\nabla^{l+1}(\nabla\Phi,u,w,b)\|^2_{L^2},
\end{equation}
Taking into account Lemma \ref{le3.1}-\ref{le3.2} and Theorem \ref{th3.1}, we have that for $k=0,1,\cdots,N-2$,
\begin{align}\label{5.6}
&\frac{1}{2}\frac{d}{dt}\sum_{l=k}^{k+2}\|\nabla^l(\nabla\Phi,\varrho,u,w,b)\|_{L^2}^{2}+\frac12\sum_{l=k}^{k+2}\|\nabla^lw\|_{L^2}^{2}
+\sum_{l=k}^{k+2}(\frac23\|\nabla^{l+1}u\|_{L^2}^{2}+3\|\nabla^{l+1}w\|_{L^2}^{2}+\|\nabla^{l+1}b\|_{L^2}^{2})\nonumber\\
&\leq c_l\varepsilon_0\sum_{l=k}^{k+2}\left(\|\nabla^l(\varrho,w)\|_{L^2}^{2}+\|\nabla^{l+1}(\nabla\Phi,u,w,b)\|_{L^2}^{2}\right).
\end{align}
and
\begin{align}\label{5.7}
&\frac{d}{dt}\sum_{l=k}^{k+1}\int \nabla^l u\cdot \nabla^{l+1}\varrho dx
+\frac{1}{4}\sum_{l=k}^{k+1}(\|\nabla^{l}\varrho\|_{L^2}^2+\|\nabla^{l+1}\varrho\|_{L^2}^2+\|\nabla^{l+1}\nabla\Phi\|_{L^2}^2+
\|\nabla^{l+1}\nabla\Phi\|_{L^2}^2)\nonumber\\
\leq&c_{l}\varepsilon_0\sum_{l=0}^{N-1}(\|\nabla^{l+1}\varrho\|_{L^2}^2+\|\nabla^{l+2}(w,b)\|_{L^2}^2)
+\sum_{l=k}^{k+1}(\|\nabla^{l+1}u\|_{L^2}^2+2\|\nabla^{l+1}w\|_{L^2}^2+4\|\nabla^{l+2}u\|_{L^2}^2).
\end{align}
By linear combination of (\ref{5.6}) and (\ref{5.7}), since $\varepsilon_0$ is small, we deduce that there exists an instant energy functional
$\tilde{\mathscr{E}}_{k}^{k+2}$ is equivalent to $\mathscr{E}_{k}^{k+2}$ such that
\begin{equation}\label{5.8}
\frac{d}{dt}\tilde{\mathscr{E}}_{k}^{k+2}+\mathscr{D}_{k}^{k+2}\leq 0.
\end{equation}
Noting that
$\mathscr{D}_{k}^{k+2}$ is weaker than $\mathscr{E}_{k}^{k+2}$, which prevents the exponent decay of the solution. We need to bound the missing
terms in the energy, that is, $\|\nabla^k(\nabla\Phi,u,b)\|^2_{L^2}$ in terms of $\mathscr{E}_{k}^{k+2}$. From which, then we can derive the time decay
rate from (\ref{5.8}). For this aim, we need the Sobolev interpolation between the negative and positive Sobolev norms. From now on, we assume for the
moment that we have proved (\ref{1.13}) and (\ref{1.14}). Using Lemma \ref{le2.6} for $s>0$ and $k+s\geq0$, we have
\begin{equation}\label{5.9}
\|\nabla^k(\nabla\Phi,u,b)\|_{L^2}\leq \|(\nabla\Phi,u,b)\|_{\dot{H}^{-s}}^{\frac{1}{k+1+s}}
\|\nabla^{k+1}(\nabla\Phi,u,b)\|_{L^2}^{\frac{k+s}{k+1+s}}\leq c\|\nabla^{k+1}(\nabla\Phi,u,b)\|_{L^2}^{\frac{k+s}{k+1+s}}.
\end{equation}
Similarly, applying Lemma \ref{le2.7}, for  $s>0$ and $k+s\geq0$, we have
\begin{equation}\label{5.10}
\|\nabla^k(\nabla\Phi,u,b)\|_{L^2}\leq \|(\nabla\Phi,u,b)\|_{\dot{B}^{-s}_{2,\infty}}^{\frac{1}{k+1+s}}
\|\nabla^{k+1}(\nabla\Phi,u,b)\|_{L^2}^{\frac{k+s}{k+1+s}}\leq c\|\nabla^{k+1}(\nabla\Phi,u,b)\|_{L^2}^{\frac{k+s}{k+1+s}}.
\end{equation}
As a consequence, from (\ref{5.8})-(\ref{5.10}), it follows that
\begin{equation}\label{5.11}
\frac{d}{dt}\mathscr{E}_{k}^{k+2}+(\mathscr{E}_{k}^{k+2})^{1+\alpha}\leq0,
\end{equation}
where $\alpha=\frac{1}{k+s}$, $k=0,1,\cdots, N-2$. Solving this inequality directly, we are in a position to obtain
\begin{equation*}
\mathscr{E}_{k}^{k+2}(t)\leq\left\{(\mathscr{E}_{k}^{k+2}(0))^{-\alpha}+\alpha t\right\}^{-\frac{1}{\alpha}}=c_0(1+t)^{-(k+s)}.
\end{equation*}
This proves the optimal decay (\ref{1.15}). On the other hand, since $\varrho=\textrm{div}\nabla\Phi$, we deduce that
\begin{equation*}
\|\nabla^{k}\varrho(t)\|^2_{L^2}\leq \|\nabla^{k+1}\nabla\Phi\|^2_{L^2}\leq c_0(1+t)^{-(k+1+s)},
\end{equation*}
whence (\ref{1.16}).

Finally, we turn back to prove (\ref{1.13}) and (\ref{1.14}). First, by Lemma \ref{le4.1}, we propose to prove (\ref{1.13}). However, we are not able to
prove it for all $s\in [0,\frac32]$ at this moment, we must distinguish the argument by the value of $s$. First, it is trivial for the case  $s=0$, now,
for $s\in (0,\frac12]$, integrating (\ref{4.1}) in time, by (\ref{1.11}), we obtain that for $s\in (0,\frac12]$
\begin{align}\label{5.12}
\|(\nabla\Phi,\varrho,u,w,b)\|^2_{\dot{H}^{-s}}&\lesssim \|(\nabla\Phi(0),\varrho_0,u_0,w_0,b_0)\|^2_{\dot{H}^{-s}}
+\int_{0}^t\mathscr{D}_3(\tau)\|(\nabla\Phi,\varrho,u,w,b)\|_{\dot{H}^{-s}}d\tau\nonumber  \\
 &\leq c(1+\sup_{0\leq\tau\leq t}\|(\nabla\Phi,\varrho,u,w,b)\|_{\dot{H}^{-s}}).
\end{align}
This, together with the  Cauchy's inequality implies (\ref{1.13}) for $s\in (0,\frac12]$ and thus verifies (\ref{1.15}) for $s\in (0,\frac12]$. Next, let
$s\in (\frac12,1)$, note that, the arguments for the case $s\in (0,\frac12]$ can not be applied to this case. However, observing that we have
$(\nabla\Phi(0),\varrho_0,u_0,w_0,b_0)\in \dot{H}^{-\frac12}$ due to the fact $\dot{H}^{-s}\cap L^2\subset \dot{H}^{-q}$ for any $q\in [0,s]$. At this
stage, from (\ref{1.15}), it holds that for $k\geq 0$ and $N\geq k+2$
\begin{equation}\label{5.13}
\|\nabla^k(\nabla\Phi,\varrho,u,w,b)(t)\|_{L^2}\leq c_0(1+t)^{-\frac{k+\frac12}{2}}.
\end{equation}
Thus, integrating (\ref{4.2}) in time for $s\in (\frac12,1)$ and applying (\ref{5.13}), yields that
\begin{align}\label{5.14}
\|(\nabla\Phi,\varrho,u,w,b)\|^2_{\dot{H}^{-s}}\lesssim &\|(\nabla\Phi(0),\varrho_0,u_0,w_0,b_0)\|^2_{\dot{H}^{-s}}
+\int_{0}^t\sqrt{\mathscr{D}_3(\tau)}\|(\varrho,u,w,b)(\tau)\|_{L^2}^{s-\frac12}\nonumber  \\
&\times\|\nabla(\varrho,u,w,b)(\tau)\|_{L^2}^{\frac32-s}
\|(\nabla\Phi,\varrho,u,w,b)\|_{\dot{H}^{-s}}d\tau\nonumber  \\
\leq& c(1+\sup_{0\leq\tau\leq t}\|(\nabla\Phi,\varrho,u,w,b)\|_{\dot{H}^{-s}}\int_0^t(1+\tau)^{-2(1-\frac s2)}d\tau)\nonumber  \\
\leq& c_0(1+\sup_{0\leq\tau\leq t}\|(\nabla\Phi,\varrho,u,w,b)\|_{\dot{H}^{-s}}),
\end{align}
where in the last inequality, we have used the fact $s\in (\frac12,1)$, so that the time integral is finite. By the Cauchy's inequality, this implies
(\ref{1.13}) for $s\in (\frac12,1)$, from which, we also verify (\ref{1.15}) for $s\in (\frac12,1)$. Finally, let $s\in [1,\frac32)$, we choose
$s_0$ such that $s-\frac12<s_0<1$. Then $(\nabla\Phi(0),\varrho_0,u_0,w_0,b_0)\in \dot{H}^{-s_0}$, and from (\ref{1.15}), the following estimate holds
\begin{equation}\label{5.15}
\|\nabla^k(\nabla\Phi,\varrho,u,w,b)\|_{L^2}\leq c_0(1+t)^{-\frac{k+s_0}{2}}
\end{equation}
for $k\geq0$ and $N\geq k+2$. Therefore, similar to (\ref{5.14}), using (\ref{5.15}) and (\ref{4.2}) for $s\in (1,\frac32)$, we conclude that
\begin{align}\label{5.16}
\|(\nabla\Phi,\varrho,u,w,b)\|^2_{\dot{H}^{-s}}
\leq& c(1+\sup_{0\leq\tau\leq t}\|(\nabla\Phi,\varrho,u,w,b)\|_{\dot{H}^{-s}}\int_0^t(1+\tau)^{s_0+\frac32-s}d\tau)\nonumber  \\
\leq& c_0(1+\sup_{0\leq\tau\leq t}\|(\nabla\Phi,\varrho,u,w,b)\|_{\dot{H}^{-s}}).
\end{align}
Here, we have taken into account $s-s_0<\frac12$, so that the time integral in (\ref{5.16}) is finite. This implies (\ref{5.13}) for $s\in (1,\frac32)$
and thus we have proved (\ref{1.15}) for $s\in (1,\frac32)$. The rest of the proof is exactly same with above, we only need to replace Lemma \ref{le2.6}
and Lemma \ref{le4.1}  by Lemma \ref{le2.7} and Lemma \ref{le4.2}, respectively. Then, we can deduce (\ref{1.14}) for $s\in (0,\frac32]$, and thus verify
(\ref{1.16}) for $s\in (0,\frac32]$. Here, we just skip it. Thus, we have completed the proof of Theorem \ref{th1.2}.
\end{proof}


\begin{thebibliography}{99}



\bibitem{va6}D.W. Condiff, J.S. Dahler, Fluid mechanics aspects of antisymmetric stress, Phys. Fluids, 7 (1964) 842-854.

\bibitem{miao}Q. Chen, C. Miao, Global well-posedness for the micropolar fluid system
in critical Besov spaces, J. Differential Equations, 252 (2012) 2698-2724.




\bibitem{miao10}B. Dong, Z. Zhang, Global regularity for the $2D$ micropolar fluid flows with zero angular viscosity, J. Differential Equations,
249 (2010) 200-213.

\bibitem{ww7}I. Dra\v{z}i\'{c}, N. Mujakovi\'{c}, N. \v{C}rnjari\'{C}-\v{Z}ic, Three-dimensional compressible viscous micropolar fluid with cylindrical symmetry: derivation of the model and a numerical solution, Math. Comput. Simulation, 140 (2017) 107-124.
\bibitem{ww8}I. Dra\v{z}i\'{c}, Three-dimensional flow of a compressible viscous micropolar fluid with cylindrical symmetry: a global existence theorem, Math. Methods Appl. Sci. 40 (13) (2017) 4785-4801.
\bibitem{ww9}I. Dra\v{z}i\'{c}, N. Mujakovi\'{c}, $3D$ flow of a compressible viscous micropolar fluid with spherical symmetry: large time behavior of the solution, J. Math. Anal. Appl. 431 (2015) 545-568.

\bibitem{duan1}	R.J. Duan, S. Ukai, T. Yang, H.J. Zhao, Optimal convergence rates for the compressible Navier-Stokes equations with potential forces, Math. Models Methods Appl. Sci. 17 (2007) 737-758.
\bibitem{duan2}	R.J. Duan, H.X. Liu, S. Ukai, T. Yang, Optimal $L^p-L^q$ convergence rates for the compressible Navier-Stokes equations with potential force, J. Differential Equations, 238 (2007) 220-233.


\bibitem{silva7}	M. Dur\'{a}n, J. Ferreira, M.A. Rojas-Medar, Reproductive weak solutions of magneto-micropolar fluid equations in exterior domains, Math. Comput. Modelling 35 (2002) 779-791.

\bibitem{gala4}	A.C. Eringen, Theory of micropolar fluids, J. Math. Mech. 16 (1966) 1-8.

\bibitem{pan}H. Frid, D. R. Marroquin, R. Pan,  Modeling aurora type phenomena by short wave-long wave interactions in multidimensional large magnetohydrodynamic flows, SIAM J. Math. Anal. 50(6) (2018) 6156-6195.

\bibitem{gala}S. Gala, Regularity criteria for the $3D$ magneto-micropolar fluid equations in the Morrey-Campanato space, NoDEA Nonlinear Differential Equations Appl. 17 (2010) 181-194.

\bibitem{twz14}	G.P. Galdi, S. Rionero, A note on the existence and uniqueness of solutions of the microploar fluid equations, Internat. J. Engrg. Sci. 15 (1977) 105-108.

\bibitem{wy4}L. Grafakos, Classical and Modern Fourier Analysis, Pearson Education, Inc., Prentice-Hall, 2004.

\bibitem{w7}Y. Guo, Smooth irrotational flows in the large to the Euler-Poisson system, Comm. Math. Phys., 195 (1998) 249-265.

\bibitem{twz16}Y. Guo, The Vlasov-Poisson-Landau system in a periodic box, J. Amer. Math. Soc., 25 (2012) 759-812.

\bibitem{twz17}Y. Guo, Y.J. Wang, Decay of dissipative equations and negative Sobolev spaces, Comm. Partial Differential Equations, 37 (2012) 2165-2208.

\bibitem{hu1}	X. Hu, D. Wang, Global solutions to the three-dimensional full compressible magnetohydrodynamic flows, Comm. Math. Phys. 283 (2008) 253-284.


\bibitem{hu}X. Hu, D. Wang,  Global existence and large-time behavior of solutions to the three-dimensional equations of compressible magnetohydrodynamic flows, Arch. Ration. Mech. Anal. 197(1) (2010) 203-238.

\bibitem{yj14}N. Ju, Existence and uniqueness of the solution to dissipative $2D$ Quasi-Geostrophic equations in the Sobolev space,
Comm. Math. Phys., 251 (2004) 365-376.


\bibitem{kaw}S. Kawashima, Smooth global solutions for two-dimensional equations of electromagneto fluid dynamics, Jpn. J. Appl. Math. 1 (1984) 207-222.

\bibitem{kaw1}	S. Kawashima, System of a Hyperbolic-Parabolic Composite Type, with Applications to the Equations of Manetohydrodynamics, thesis, Kyoto University, Kyoto, 1983.
\bibitem{kob}T. Kobayashi, Y. Shibata, Decay estimates of solutions for the equations of motion of compressible viscous and heat-conductive gases in an exterior domain in $\mathbb{R}^3$, Comm. Math. Phys. 200 (1999) 621-659.



\bibitem{br12}	I. Kondrashuk, E.A. Notte-Cuello, M.A. Rojas-Medar, Stationary asymmetric fluids and Hodge operator, Bol. Soc. Esp. Mat. Apl. SeMA 47 (2009) 99-106.

\bibitem{lyj}Y. Liu, S. Li, Global well-posedness for magneto-micropolar system in $2\frac12$ dimensions, Appl. Math. Comput. 280 (2016) 72-85.

\bibitem{lx}	T.P. Liu, Z.P. Xin, Nonlinear stability of rarefaction waves for compressible Navier-Stokes equations, Comm. Math. Phys. 118 (1988) 451-465.


\bibitem{ww21}	Q.Q. Liu, P. Zhang, Optimal time decay of the compressible micropolar fluids, J. Differential Equations, 260(10) (2016) 7634-7661.

\bibitem{ww23}Q.Q. Liu, P. Zhang, Long time behavior of solution to the compressible micropolar fluids with external force, Nonlinear Anal. Real World Appl. 40 (2018) 361-376.



\bibitem{loa}M. Loayza,  M.A. Rojas-Medar, A weak-$L^p$ Prodi-Serrin type regularity criterion for the micropolar fluid equations,
J. Math. Phys. 57(2) (2016) 1-7.

\bibitem{ma4}G. {\L}ukaszewicz, Micropolar Fluids: Theorey and Applications, Birkh\"{a}user, Boston, 1999.

\bibitem{br16}	G. {\L}ukaszewicz, M.A. Rojas-Medar, M.A. Santos, Stationary micropolar fluid with boundary data in $L^2$, J. Math. Anal. Appl. 271 (2002) 91-107.

\bibitem{sadowski}G. {\L}ukaszewicz, W. Sadowski, Uniform attractor for $2D$ magneto-micropolar fluid flow in some unbounded domains, Z. Angew. Math. Phys. 55 (2004) 247-257.


\bibitem{ma}L. Ma, On two-dimensional incompressible magneto-micropolar system with mixed partial viscosity,  Nonlinear Anal. Real World Appl. 40 (2018) 95-129.

\bibitem{mat1}	A. Matsumura, T. Nishida, The initial value problem for the equations of motion of compressible viscous and heat-conductive fluids, Proc. Japan Acad. Ser. A, 55 (1979) 337-342.
\bibitem{mat2}	A. Matsumura, T. Nishida, The initial value problem for the equations of motion of viscous and heat-conductive gases, J. Math. Kyoto Univ. 20 (1980) 67-104.


\bibitem{ww30}	N. Mujakovi\'{c}, One-dimensional flow of a compressible viscous micropolar fluid: regularity of the solution, Rad Mat. 10 (2) (2001) 181-193.
\bibitem{ww31}	N. Mujakovi\'{c}, Global in time estimates for one-dimensional compressible viscous micropolar fluid model, Glas. Mat. Ser. III, 40 (1) (2005) 103-120.
\bibitem{ww32}	N. Mujakovi\'{c}, One-dimensional flow of a compressible viscous micropolar fluid: the Cauchy problem, Math. Commun. 10 (1) (2005) 1-14.
\bibitem{ww33}	N. Mujakovi\'{c}, Non-homogeneous boundary value problem for one-dimensional compressible viscous micropolar fluid model: a local existence theorem, Ann. Univ. Ferrara Sez. VII Sci. Mat. 53 (2) (2007) 361-379.



\bibitem{y60}L. Nirenberg, On elliptic partial differential equations, Ann. Sc. Norm. Super. Pisa Cl. Sci.,
13 (1959) 115-162.


\bibitem{gala17}E.E. Ortega-Torres, M.A. Rojas-Medar,  Magneto-micropolar fluid motion:
global existence of strong solutions, Abstr. Appl. Anal. 4 (1999) 109-125.

\bibitem{gala18}M.A. Rojas-Medar, Magneto-micropolar fuid motion: existence and uniqueness of strong solutions, Math. Nachr. 188 (1997) 301-319.

\bibitem{yao2}	M.A. Rojas-Medar, J.L. Boldrini, Magneto-microploar fluid motion: existence of weak solutions, Internet. Rev. Mat. Complut. 11 (1998) 443-460.

\bibitem{twz36}V. Sohinger, R. M. Strain, The Boltzmann equation, Besov spaces, and optimal time decay rates in $\mathbb{R}^n_x$, Advances in Mathematics,
261 (2014) 274-332.

\bibitem{wyj31}E. M. Stein, Singular Integrals and Differentiability Properties of Functions, Princeton University Press, 1970.



\bibitem{twz38}R. M. Strain, Y. Guo, Almost exponential decay near Maxwellian, Comm. Partial Differential Equations, 31 (2006) 417-429.



\bibitem{zhou}Z. Tan, W. Wu, J. Zhou, Global existence and decay estimate of solutions to magneto-micropolar fluid equations, J. Differential Equations, 266 (2019) 4137-4169.

\bibitem{ume}	T. Umeda, S. Kawashima, Y. Shizuta, On the decay of solutions to the linearized equations of electromagneto fluid dynamics, Jpn. J. Appl. Math. 1 (1984) 435-457.
	
\bibitem{br30} 	E.J. Villamizar-Roa, M.A. Rodriguez-Bellido, Global existence and exponential stability for the micropolar fluids, Z. Angew. Math. Phys. 59 (2008) 790-809.

\bibitem{vol}A.I. Vol'pert, S.I. Hudiaev, On the Cauchy problem for composite systems of nonlinear equations, Mat. Sb. 87 (1972) 504-528.

\bibitem{wlt}Y. Wang, C. Liu, Z. Tan,  A generalized Poisson-Nernst-Planck-Navier-Stokes model on the fluid with the crowded charged particles: derivation and its well-posedness. SIAM J. Math. Anal. 48(5) (2016),  3191-3235.


\bibitem{twz45} R. Wei,  B. Guo, Y. Li,  Global existence and optimal convergence rates of solutions for 3D compressible magneto-micropolar fluid equations, J. Differential Equations, 263(5) (2017) 2457-2480.



\bibitem{wuw}Z. Wu, W. Wang, The pointwise estimates of diffusion wave of the compressible micropolar fluids, J. Differential Equations, 265 (2018) 2544-2576.
\bibitem{wjh}J. Wu, Y. Wu,  Global small solutions to the compressible $2D$ magnetohydrodynamic system without magnetic diffusion, Adv. Math. 310 (2017) 759-888.

\bibitem{gala23}N. Yamaguchi,  Existence of global strong solution to the micropolar Fuid system in a bounded domain,  Math. Meth. Appl. Sci. 28  (2005) 1507-1526.

\bibitem{miao18}B. Yuan, On the regularity criteria for weak solutions to the micropolar fluid equations in Lorentz space, Proc. Amer. Math.
Soc. 138 (2010) 2025-2036.
\bibitem{miao19} J. Yuan, Existence theorem and blow-up criterion of the strong solutions to the magneto-micropolar fluid equations, Math.
Methods Appl. Sci. 31 (2008) 1113-1130.

\bibitem{yao} Z. Zhang, Z.-A. Yao, X. Wang, A regularity criterion for the $3D$ magneto-microloar fluid equations in Triebel-Lizorkin spaces, Nonlinear Anal. 74 (2011) 2220-2225.








\end{thebibliography}
\end{document}